\documentclass[11pt,a4paper,reqno]{amsart}
\usepackage[margin=2cm]{geometry}
\usepackage{hyperref,amssymb,amsfonts,amsthm,amsrefs,mathtools,verbatim,microtype,bm}
\usepackage{stmaryrd}
\usepackage[shortlabels]{enumitem}

\newtheorem{Theorem}{Theorem}[section]
\newtheorem{Proposition}[Theorem]{Proposition}
\newtheorem{Lemma}[Theorem]{Lemma}
\newtheorem{Corollary}[Theorem]{Corollary}
\newtheorem*{OldMainTheorem}{Theorem~\ref{thm-BCSigma2ShuffleOld}}
\newtheorem*{NewMainTheorem}{Theorem~\ref{thm-BCSigma2ShuffleNew}}
\newtheorem{ClaimP}{Claim}
\newtheorem{ClaimC}{Claim}

\newtheorem*{Claim*}{Claim}

\theoremstyle{definition}
\newtheorem{Definition}[Theorem]{Definition}
\newtheorem{Example}[Theorem]{Example}

\newtheorem{Remark}[Theorem]{Remark}

\DeclareFontFamily{U}{mathb}{\hyphenchar\font45}
\DeclareFontShape{U}{mathb}{m}{n}{
<-6> mathb5 <6-7> mathb6 <7-8> mathb7
<8-9> mathb8 <9-10> mathb9
<10-12> mathb10 <12-> mathb12
}{}
\DeclareSymbolFont{mathb}{U}{mathb}{m}{n}
\DeclareMathSymbol{\pprec}{\mathrel}{mathb}{"CE}
\DeclareMathSymbol{\ssucc}{\mathrel}{mathb}{"CF}

\DeclareFontFamily{U}{mathb}{\hyphenchar\font45}
\DeclareFontShape{U}{mathb}{m}{n}{
      <5> <6> <7> <8> <9> <10> gen * mathb
      <10.95> mathb10 <12> <14.4> <17.28> <20.74> <24.88> mathb12
      }{}
\DeclareSymbolFont{mathb}{U}{mathb}{m}{n}
\DeclareFontSubstitution{U}{mathb}{m}{n}
\DeclareMathSymbol{\monus}{2}{mathb}{"01}

\DeclareMathOperator{\dom}{\mathrm{dom}}
\DeclareMathOperator{\ran}{\mathrm{ran}}

\DeclareMathOperator{\condF}{\mathbf{c}_\textup{F}}

\DeclareMathOperator{\ct}{\mathrm{ct}}

\newcommand{\la}{\langle}
\newcommand{\ra}{\rangle}
\newcommand{\da}{{\downarrow}}
\newcommand{\ua}{{\uparrow}}
\newcommand{\imp}{\rightarrow}
\newcommand{\Imp}{\Rightarrow}

\newcommand{\Biimp}{\Leftrightarrow}
\newcommand{\llb}{\llbracket}
\newcommand{\rrb}{\rrbracket}
\newcommand{\iso}{\cong}
\newcommand{\niso}{\ncong}
\newcommand{\Nb}{\mathbb{N}}
\newcommand{\Zb}{\mathbb{Z}}
\newcommand{\Qb}{\mathbb{Q}}

\newcommand{\rst}{{\restriction}}
\newcommand{\keq}{\simeq}
\newcommand{\forae}{\forall^\infty}
\newcommand{\existsinf}{\exists^\infty}
\newcommand{\fct}{\mathrm{ct}}

\newcommand{\mc}[1]{\mathcal{#1}}
\newcommand{\mf}[1]{\mathfrak{#1}}
\newcommand{\ol}[1]{\overline{#1}}
\newcommand{\ora}[1]{\overrightarrow{#1}}
\newcommand{\wh}[1]{\widehat{#1}}

\newcommand{\leqT}{\leq_\mathrm{T}}

\newcommand{\geqT}{\geq_\mathrm{T}}

\AtBeginDocument{%
   \def\MR#1{}
}

\title[Effective powers of $\omega$ over $\Delta_2$ cohesive sets]{Effective powers of \texorpdfstring{$\omega$}{omega} over \texorpdfstring{$\Delta_2$}{Delta\_2} cohesive sets and infinite \texorpdfstring{$\Pi_1$}{Pi\_1} sets without \texorpdfstring{$\Delta_2$}{Delta\_2} cohesive subsets}

\author{Paul Shafer}
\address{School of Mathematics\\
University of Leeds\\
Leeds\\
LS2 9JT\\
United Kingdom}
\email{p.e.shafer@leeds.ac.uk}
\urladdr{https://peshafer.github.io}
\urladdr{https://orcid.org/0000-0001-5386-9218}

\date{\today}

\begin{document}

\begin{abstract}
Recall that an infinite subset of $\Nb$ is \emph{cohesive} if it cannot be split into two infinite pieces by a c.e.\ set.  Likewise, a set is \emph{r-cohesive} if it cannot be split by a computable set, and it is \emph{p-cohesive} if it cannot be split by a primitive recursive set.  A \emph{cohesive power} of a computable structure is an effective ultrapower of the structure, where a cohesive set plays the role of an ultrafilter.  Let $\omega$, $\zeta$, and $\eta$ denote the respective order-types of the natural numbers, the integers, and the rationals.  We study cohesive powers of computable copies of $\omega$ over $\Delta_2$ cohesive sets.  We show that there is a computable copy $\mc{L}$ of $\omega$ such that, for every $\Delta_2$ cohesive set $C$, the cohesive power of $\mc{L}$ over $C$ has order-type $\omega + \eta$.  This improves an earlier result of Dimitrov, Harizanov, Morozov, Shafer, A.~Soskova, and Vatev by generalizing from $\Pi_1$ cohesive sets to $\Delta_2$ cohesive sets and by computing a single copy of $\omega$ that has the desired cohesive power over all $\Delta_2$ cohesive sets.  Furthermore, our result is optimal in the sense that $\Delta_2$ cannot be replaced by $\Pi_2$.  More generally, we show that if $X \subseteq \Nb \setminus \{0\}$ is a Boolean combination of $\Sigma_2$ sets, thought of as a set of finite order-types, then there is a computable copy $\mc{L}$ of $\omega$ where the cohesive power of $\mc{L}$ over any $\Delta_2$ cohesive set has order-type $\omega + \bm{\sigma}(X \cup \{\omega + \zeta\eta + \omega^*\})$.  If $X$ is finite and non-empty, then there is also a computable copy $\mc{L}$ of $\omega$ where the cohesive power of $\mc{L}$ over any $\Delta_2$ cohesive set has order-type $\omega + \bm{\sigma}(X)$.  Here $\bm{\sigma}$ denotes the \emph{shuffle} operation.  An unexpected byproduct of our work is a new method for constructing infinite $\Pi_1$ sets that do not have $\Delta_2$ cohesive subsets.  In fact, we construct an infinite $\Pi_1$ set that does not have a $\Delta_2$ p-cohesive subset.  Infinite $\Pi_1$ sets without $\Delta_2$ r-cohesive subsets generalize D.\ Martin's classic co-infinite c.e.\ set with no maximal superset and have appeared in the work of Lerman, Shore, and Soare.
\end{abstract}

\maketitle

\section{Introduction}

Ultraproducts and ultrapowers are tools from mathematical logic with a wide variety of applications throughout mathematics \cites{BankstonSurvey, DiNassoGoldbringMartino, EklofSurvey, HensonIovinoSurvey, KeislerSurvey, MalliarisSurvey}.  The motivating idea behind the ultrapower construction originates with Skolem's construction of a countable non-standard model of arithmetic.  Let \mbox{$\mf{L} = (0, 1, <, +, \times)$} denote the language of arithmetic.  Let $\mc{N} = (\Nb; 0, 1, <, +, \times)$ denote the standard model of arithmetic, which has universe $\Nb$ and interprets the symbols $0$, $1$, $<$, $+$, and $\times$ in the usual ways.  Skolem's construction produces a countable $\mf{L}$-structure $\mc{M}$ that is \emph{elementarily equivalent} to $\mc{N}$ (i.e., satisfies the same first-order $\mf{L}$-sentences as $\mc{N}$) but is not isomorphic to $\mc{N}$.  It follows that $\mc{N}$ cannot be characterized among the countable $\mf{L}$-structures on the basis of first-order statements alone.

Skolem's construction works as follows.  For sets $X, Y \subseteq \Nb$, let $X \subseteq^* Y$ denote that $X \setminus Y$ is finite, and let $\ol{X} = \Nb \setminus X$ denote the complement of $X$.  It is not hard to show that for any countable sequence $\vec{A} = (A_n : n \in \Nb)$ of subsets of $\Nb$, there is an infinite set $C \subseteq^* \Nb$ such that for every $n$, either $C \subseteq^* A_n$ or $C \subseteq^* \ol{A}_n$.  Such a set $C$ is called \emph{cohesive} for $\vec{A}$, or simply \emph{$\vec{A}$-cohesive}.  There are only countably many arithmetically definable sets, so we may fix a set $C$ that is cohesive for this collection.  Thus for every arithmetical set $A$, either $C \subseteq^* A$ or $C \subseteq^* \ol{A}$.  Now consider the arithmetically definable functions $f \colon \Nb \to \Nb$.  Given two such functions $f$ and $g$, the set $\{n : f(n) = g(n)\}$ is arithmetical, so either $C \subseteq^* \{n : f(n) = g(n)\}$ or $C \subseteq^* \{n : f(n) \neq g(n)\}$.  Define $f =_C g$ if $C \subseteq^* \{n : f(n) = g(n)\}$, and notice that $=_C$ is an equivalence relation.  Let $[f]$ denote the $=_C$-equivalence class of $f$, and define an $\mf{L}$-structure $\mc{M}$ on the $=_C$-equivalence classes in the following way.  Interpret $0$ and $1$ as the equivalence classes of the constant functions $n \mapsto 0$ and $n \mapsto 1$.  Interpret $<$ by letting $[f] < [g]$ if and only if $C \subseteq^* \{n : f(n) < g(n)\}$.  Interpret $+$ and $\times$ by letting $[f] + [g] = [f + g]$ and $[f] \times [g] = [f \times g]$, where $f + g$ and $f \times g$ are computed pointwise.  One then shows that $\mc{M}$ is elementarily equivalent to $\mc{N}$.  Moreover, $\mc{M}$ is countable because there are only countably many arithmetical functions, and $\mc{M}$ is not isomorphic to $\mc{N}$ because $\mc{M}$ contains elements with infinitely many predecessors, such as the element represented by the identity function.

Nowadays we may think of Skolem's construction as effectivizing the ultrapower construction, where $C \subseteq^* A$ plays the role of $A$ being in the ultrafilter.  Skolem's construction only considers functions $f \colon \Nb \to \Nb$ of arithmetical complexity and only requires the $C \subseteq^* A$ or $C \subseteq^* \ol{A}$ dichotomy for sets $A$ of arithmetical complexity.  Restricting the complexity of the sequence $\vec{A}$ also induces a bound on the complexity that suffices to produce an $\vec{A}$-cohesive set $C$.  It becomes natural to ask what happens to Skolem's construction when one restricts to even lower complexities.  Call a set $C$:
\begin{itemize}
\item \emph{p-cohesive} if $C$ is cohesive for the collection of primitive recursive sets,
\smallskip
\item \emph{r-cohesive} if $C$ is cohesive for the collection of computable (aka recursive) sets, and
\smallskip
\item \emph{cohesive} if $C$ is cohesive for the collection of computably enumerable (aka recursively enumerable) sets.
\end{itemize}
Again let $\mc{N}$ denote the standard model of arithmetic.  Feferman, Scott, and Tennenbaum~\cite{FefermanScottTennenbaum} consider the effective ultrapowers of $\mc{N}$ obtained by restricting to computable functions $f$ and by using r-cohesive sets $C$ to take the powers.  They show that the resulting structures are never elementarily equivalent to $\mc{N}$.  In fact, they show that the resulting structures are never even models of Peano arithmetic.  Works such as~\cites{HirschfeldModels, HirschfeldWheelerBook, ShavrukovPrimePowers, McLaughlinRearrangement, McLaughlinCollectionFails, McLaughlinSurvey, McLaughlinUltrapowersModelsExtensions, McLaughlinTotallyRigid, McLaughlinEmbeddings, LermanCo-r-Max} continue studying flavors of effective ultrapowers of the particular structure $\mc{N}$, including powers over cohesive sets as discussed here as well as the so-called \emph{recursive ultrapowers} and \emph{r.e.\ ultrapowers}.  There is a particular focus on rigidity.  Hirschfeld and Wheeler~\cite{HirschfeldWheelerBook} show that the r.e.\ ultrapowers of $\mc{N}$ are rigid, McLaughlin~\cite{McLaughlinRearrangement} observes that the same holds of the recursive ultrapowers of $\mc{N}$, and later McLaughlin~\cite{McLaughlinTotallyRigid} shows that the recursive ultrapowers of $\mc{N}$ are totally rigid, meaning that they have no non-trivial isomorphic self-embeddings.  This line culminates in recent work of Shavrukov~\cite{ShavrukovPrimePowers}, which shows that the r.e.\ ultrapowers of $\mc{N}$ are totally rigid as well.  Shavrukov also introduces the \emph{r.e.\ prime powers} of $\mc{N}$ (which he shows are equivalent to the powers of $\mc{N}$ over cohesive sets) and shows that they need not always be totally rigid.

Dimitrov~\cite{DimitrovCohPow} generalizes from the particular structure $\mc{N}$ to an arbitrary computable structure $\mc{A}$, and he calls the effective ultrapower of $\mc{A}$ over a cohesive set $C$ the \emph{cohesive power} of $\mc{A}$ over $C$.  Cohesive powers of computable structures find applications to the lattice of c.e.\ subspaces of computable vector spaces in~\cites{DimitrovHarizanov, DimitrovLattices, DimitrovHarizanovMillerMourad}.  Rigidity is studied as well.  For example, Dimitrov, Harizanov, R.~Miller, and Mourad show that cohesive powers of the field of rational numbers over $\Pi_1$ cohesive sets are rigid~\cite{DimitrovHarizanovMillerMourad}.

In the classical case of the ultrapower of a structure $\mc{A}$ over an ultrafilter $\mc{U}$ (say on $\Nb$), the resulting structure may depend on the choice of the ultrafilter.  Consider, for example, the structure $(\Nb; <)$ of the natural numbers as a linear order.  If the continuum hypothesis fails, then there are $2^{2^{\aleph_0}}$ pairwise non-isomorphic linear orders that arise as ultrapowers of $(\Nb; <)$ over ultrafilters on $\Nb$~\cite{KramerShelahTentThomas}.  In the effective case, Lerman~\cite{LermanCo-r-Max} shows that when considering a cohesive power of $\mc{N}$ over a $\Pi_1$ cohesive set $C$, the resulting structure depends only on the many-one degree of $C$.  Let $C$ and $D$ be $\Pi_1$ cohesive sets.  If $C$ and $D$ are many-one equivalent, then the cohesive powers of $\mc{N}$ over $C$ and $D$ are isomorphic.  If $C$ and $D$ are not many-one equivalent, then the cohesive powers of $\mc{N}$ over $C$ and $D$ are not even elementarily equivalent.

Classically, if $f$ is an isomorphism between two structures $\mc{A}$ and $\mc{B}$ and $\mc{U}$ is an ultrafilter, then $f$ induces an isomorphism between the ultrapowers of $\mc{A}$ and $\mc{B}$ over $\mc{U}$.  In the effective case, if $f$ is a \emph{computable} isomorphism between two computable structures $\mc{A}$ and $\mc{B}$ and $C$ is a cohesive set, then $f$ induces an isomorphism between the cohesive powers of $\mc{A}$ and $\mc{B}$ over $C$.  It is essential that the isomorphism $f$ is computable.  In~\cite{CohPowJournal} it is shown that for every cohesive set $C$, there are computable structures $\mc{A}$ and $\mc{B}$ that are isomorphic (but not computably isomorphic) and are such that the cohesive powers of $\mc{A}$ and $\mc{B}$ over $C$ are not elementarily equivalent.  Indeed, $\mc{A}$ and $\mc{B}$ may be taken to be isomorphic to $(\Nb; <)$.

If we assume some level of effectivity of a cohesive set $C$, then it is possible to gain finer control over cohesive powers over $C$.  This leads to many examples of non-elementarily equivalent cohesive powers of isomorphic computable structures.  Let $\omega$, $\zeta$, and $\eta$ denote the order-types of the natural numbers, the integers, and the rationals, and let $+$ and $\cdot$ denote the usual sum and product of linear orders.  It is not hard to show that the cohesive power of the usual presentation $(\Nb; <)$ of $\omega$ over any cohesive set is a linear order of type $\omega + \zeta\eta$.  This is the expected outcome because $\omega + \zeta\eta$ is familiar as the order-type of countable non-standard models of Peano arithmetic.  On the other hand, in~\cite{CohPowJournal} it is shown that for every $\Pi_1$ cohesive set $C$, there is a computable linear order $\mc{L}$ of type $\omega$ such that the cohesive power of $\mc{L}$ over $C$ is a linear order of type $\omega + \eta$.  Thus $(\Nb; <)$ and $\mc{L}$ are isomorphic linear orders with non-elementarily equivalent cohesive powers over $C$.  In fact, given any $\Pi_1$ cohesive set $C$, there are countably many computable copies of $\omega$ whose cohesive powers over $C$ yield pairwise non-elementarily equivalent linear orders.  Let $\bm{\sigma}(X)$ denote the \emph{shuffle} of a non-empty and at-most-countable set of order-types $X$ (see Definition~\ref{def-Shuffle}), and let $\omega^*$ denote the reverse of $\omega$.  The most general result of~\cite{CohPowJournal} is the following, which we discuss further in Section~\ref{sec-CohPowIntro}.

\begin{OldMainTheorem}[\cite{CohPowJournal}*{Theorem~6.10}]
Let $X \subseteq \Nb \setminus \{0\}$ be a Boolean combination of $\Sigma_2$ sets, thought of as a set of finite order-types.  Let $C$ be a $\Pi_1$ cohesive set.  Then there is a computable copy $\mc{L}$ of $\omega$ where the cohesive power of $\mc{L}$ over $C$ has order-type $\omega + \bm{\sigma}(X \cup \{\omega + \zeta\eta + \omega^*\})$.  Moreover, if $X$ is finite and non-empty, then there is also a computable copy $\mc{L}$ of $\omega$ where the cohesive power of $\mc{L}$ over $C$ has order-type $\omega + \bm{\sigma}(X)$.
\end{OldMainTheorem}

In Theorem~\ref{thm-BCSigma2ShuffleOld}, the $\Pi_1$ cohesive set $C$ is fixed in advance, and the computable copy $\mc{L}$ of $\omega$ is built so that the cohesive power of $\mc{L}$ over $C$ has a particular order-type.  The goal of this work is to improve Theorem~\ref{thm-BCSigma2ShuffleOld} to a form that is optimal with respect to the complexity of $C$.  First, we generalize from $\Pi_1$ cohesive sets to $\Delta_2$ cohesive sets.  Second, and most significantly, we swap the quantifiers on $\mc{L}$ and $C$ by showing that there is a fixed computable copy $\mc{L}$ of $\omega$ for which the cohesive power of $\mc{L}$ over $C$ has the desired order-type for \emph{every} $\Delta_2$ cohesive set $C$.  Our main result is the following.

\begin{NewMainTheorem}
Let $X \subseteq \Nb \setminus \{0\}$ be a Boolean combination of $\Sigma_2$ sets, thought of as a set of finite order-types.  Then there is a computable copy $\mc{L}$ of $\omega$ where the cohesive power of $\mc{L}$ over any $\Delta_2$ cohesive set has order-type $\omega + \bm{\sigma}(X \cup \{\omega + \zeta\eta + \omega^*\})$.  Moreover, if $X$ is finite and non-empty, then there is also a computable copy $\mc{L}$ of $\omega$ where the cohesive power of $\mc{L}$ over any $\Delta_2$ cohesive set has order-type $\omega + \bm{\sigma}(X)$.
\end{NewMainTheorem}

In particular, there is a computable copy $\mc{L}$ of $\omega$ such that the cohesive power of $\mc{L}$ over any $\Delta_2$ cohesive set has order-type $\omega + \eta$.  This is the optimal statement of this form because for every computable copy $\mc{L}$ of $\omega$, there is a $\Pi_2$ cohesive set $C$ such that the cohesive power of $\mc{L}$ over $C$ is not isomorphic to $\omega + \eta$ (see Proposition~\ref{prop-SuccNonstdPi2} below).  Furthermore, Theorem~\ref{thm-BCSigma2ShuffleNew} can be used to generate a sequence of computable linear orders $\mc{L}^1, \mc{L}^2, \mc{L}^3, \dots$ all of type $\omega$ where $\prod_C \mc{L}^k$ and $\prod_D \mc{L}^m$ are not elementarily equivalent whenever $k \neq m$ and $C$ and $D$ are $\Delta_2$ cohesive sets (see Example~\ref{ex-NonEquivSeq} below).  Our focus on computable linear orders of type $\omega$ may seem narrow at first, but examples like this explain the point.  Unusual phenomena, such as isomorphic structures having non-elementarily equivalent powers, that do not occur in the classical setting may occur in the effective setting even for the most basic and familiar mathematical structures.

An unexpected byproduct of Theorem~\ref{thm-BCSigma2ShuffleNew} is a new method for constructing infinite $\Pi_1$ sets that do not have $\Delta_2$ cohesive subsets.  In fact, we construct an infinite $\Pi_1$ set that has no $\Delta_2$ p-cohesive subset.  A classic construction of D.\ Martin~\cite{Martin} first produced an infinite $\Pi_1$ set with no $\Pi_1$ cohesive subset.  In connection to questions about the automorphisms of the lattice of c.e.\ sets, Lerman, Shore, and Soare~\cite{LermanShoreSoare} later produced an infinite $\Pi_1$ set with the property of having no $\Delta_2$ r-cohesive subset, though this property was not noted explicitly at the time.  Recently, Shavrukov~\cite{ShavrukovComm} indicated further examples of infinite $\Pi_1$ sets without $\Delta_2$ r-cohesive subsets which are connected to the notion of a \emph{single-sky} prime filter in the lattice of c.e.\ sets modulo finite difference~\cite{ShavrukovDuality}.  In terms of the arithmetical hierarchy, an infinite $\Pi_1$ set having no $\Delta_2$ p-cohesive/r-cohesive/cohesive subset is optimal because every infinite $\Pi_1$ set has a $\Pi_2$ cohesive subset.  We also observe that there are infinite $\Pi_1$ sets that do not have $\Pi_1$ r-cohesive subsets but do have $\Delta_2$ cohesive subsets.

This article is organized as follows.  In Section~\ref{sec-CohPowIntro}, we recall basic computability theory, introduce cohesive products and powers, and summarize the results of~\cite{CohPowJournal}.  We also improve~\cite{CohPowJournal}*{Theorem~2.18} by showing that the cohesive product of a uniformly computable sequence of structures over a $\Delta_2$ cohesive set is always $\Sigma_1$-recursively saturated (Theorem~\ref{thm-SatDelta2}).  In Section~\ref{sec-NoCOH}, we isolate a direct construction of an infinite $\Pi_1$ set with no $\Delta_2$ p-cohesive subset.  We feel that the construction and its connection to the lattice of c.e.\ sets is of sufficient independent interest to make a self-contained discussion worthwhile.  Section~\ref{sec-NoCOH} also serves to introduce some of the techniques that are used to prove Theorem~\ref{thm-BCSigma2ShuffleNew}.  Finally, we prove Theorem~\ref{thm-BCSigma2ShuffleNew} in Section~\ref{sec-CohPowDelta2}.

\section{Computability, cohesive products and powers, and computable linear orders}\label{sec-CohPowIntro}

We present the necessary background material concerning computability, computable structures, and cohesive products and powers, with emphasis on computable linear orders.  We refer the reader to~\cites{LermanBook, RogersBook, SoareBookRE} for further background on computability theory, to ~\cites{AshKnightBook, MontalbanBook} for further background on computable structure theory, and to~\cite{RosBook} for further background on linear orders.

Our notation mostly follows that of the sources cited above.  Throughout, $\Nb$ denotes the natural numbers, especially when used as a domain of computation, and $\omega$ denotes the order-type of the linear order $(\Nb; <)$.  For each $n \geq 1$, $\la x_0, \dots, x_{n-1} \ra \colon \Nb^n \to \Nb$ denotes the usual computable bijective tupling function that is increasing in all coordinates.  For each $i < n$, $\pi_i$ denotes the corresponding computable projection function onto coordinate $i$ given by $\pi_i(\la x_0, \dots, x_{n-1} \ra) = x_i$.  Additionally, $\Nb^{<\Nb}$ denotes the set of finite sequences over $\Nb$, which has a computable bijective encoding over $\Nb$.  Partial computable functions are denoted by $\varphi$, $\psi$, etc.  For a partial computable function $\varphi$, $\varphi(n)\da$ denotes that $\varphi$ halts on input $n$ and produces an output, and $\varphi(n)\ua$ denotes that $\varphi$ does not halt on input $n$.  Let $(\varphi_e)_{e \in \Nb}$ denote the usual effective enumeration of all partial computable functions, and let $\varphi_{e,s}(n)$ denote the result (if any) of running $\varphi_e$ on input $n$ for $s$ computational steps.  Let $W_e = \dom(\varphi_e) = \{n : \varphi_e(n)\da\}$ denote the domain of $\varphi_e$ for each $e$.  A subset of $\Nb$ is \emph{computably enumerable} (\emph{c.e.}) if it is $W_e$ for some $e$.  These are exactly the sets that are $\Sigma_1$-definable in arithmetic.   Formally we consider functions $f \colon \Nb \to \Nb$, but via tupling and sequence encoding we may interpret any such function as a function $\Nb^m \to \Nb^n$ for some $m$ and $n$, as a function $\Nb^{<\Nb} \to \Nb^{<\Nb}$, and similar.  We also identify subsets of $\Nb$ with their characteristic functions as usual.  For functions $f,g \colon \Nb \to \Nb$, $f \leqT g$ denotes that \emph{$f$ Turing reduces to / is computable from $g$}, and $f'$ denotes the \emph{Turing jump} of $f$.  Recall Post's theorem, that a subset $A \subseteq \Nb$ is $\Delta_{n+1}$-definable in arithmetic relative to $f \colon \Nb \to \Nb$ if and only if $A \leqT f^{(n)}$, the $n$\textsuperscript{th} jump of $f$ (where $f^{(0)} = f$).  In particular, a set is $\Delta_2$-definable if and only if it is computable from $0'$, the jump of the empty set.  For a function $g \colon \Nb^2 \to \Nb$, say that $\lim_s g(n, s) = y$ if $\exists s_0 \; \forall s \geq s_0 \; (g(n, s) = y)$; and say that \emph{$\lim_s g(n, s)$ exists} if $\exists y\, (\lim_s g(n, s) = y)$.  Say that a function $g \colon \Nb^2 \to \Nb$ \emph{approximates} a function $h \colon \Nb \to \Nb$ if $\forall n\, (\lim_s g(n, s) = h(n))$.  The \emph{limit lemma} (see~\cite{SoareBookRE}*{Lemma~III.3.3}) says that for every $f, h \colon \Nb \to \Nb$, $h \leqT f'$ if and only if there is a $g \colon \Nb^2 \to \Nb$ with $g \leqT f$ that approximates $h$.  In this situation, we say that \emph{$g$ is a $\Delta_2$-approximation to $h$ relative to $f$}.

\subsection*{Computable structures, cohesive products, and cohesive powers}

Fix a computable language $\mf{L}$.  A \emph{computable $\mf{L}$-structure} $\mc{A}$ consists of a non-empty computable domain $A \subseteq \Nb$ and uniformly computable interpretations of the relation, function, and constant symbols of $\mf{L}$.  Often the domain of a structure $\mc{A}$ (computable or not) is denoted by $|\mc{A}|$.  Likewise, a \emph{uniformly computable sequence of $\mf{L}$-structures} $(\mc{A}_i : i \in \Nb)$ consists of a uniformly computable sequence $(A_i : i \in \Nb)$ of non-empty domains, where $A_i \subseteq \Nb$ for each $i$, along with uniformly computable interpretations all the symbols of $\mf{L}$ in the structures $(\mc{A}_i : i \in \Nb)$.  Equivalently, an $\mf{L}$-structure $\mc{A}$ is computable if its atomic diagram is computable, and a sequence of $\mf{L}$-structures $(\mc{A}_i : i \in \Nb)$ is uniformly computable if the corresponding sequence of atomic diagrams is uniformly computable.  More generally, a computable $\mf{L}$-structure is \emph{decidable} (\emph{$n$-decidable}) if its elementary diagram ($\Sigma_n$-elementary diagram) is computable, and a uniformly computable sequence of $\mf{L}$-structures is \emph{uniformly decidable} (\emph{uniformly $n$-decidable}) if the corresponding sequence of elementary diagrams ($\Sigma_n$-elementary diagrams) is uniformly computable.  Note that a $0$-decidable $\mf{L}$-structure is the same thing as a computable $\mf{L}$-structure, and a uniformly $0$-decidable sequence of $\mf{L}$-structures is the same thing as a uniformly computable sequence of $\mf{L}$-structures.

We present the definition of \emph{cohesive products} and \emph{cohesive powers} as in~\cite{CohPowJournal}.  See also~\cites{DimitrovCohPow, DimitrovHarizanovSurvey}.

\begin{Definition}\label{def-CohProd}
Let $\mf{L}$ be a computable language.  Let $(\mc{A}_n : n \in \Nb)$ be a uniformly computable sequence of $\mf{L}$-structures with corresponding uniformly computable sequence of non-empty domains $(|\mc{A}_n| : n \in \Nb)$.  Let $C \subseteq \Nb$ be a cohesive set.  The \emph{cohesive product of $(\mc{A}_n : n \in \Nb)$ over $C$} is the $\mf{L}$-structure $\prod_C \mc{A}_n$ defined as follows.

\begin{itemize}
\item Let $D$ be the set of partial computable functions $\varphi$ such that $\forall n \, (\varphi(n)\da \,\imp\, \varphi(n) \in |\mc{A}_n|)$ and $C \subseteq^* \dom(\varphi)$.

\medskip

\item For $\varphi, \psi \in D$, let $\varphi =_C \psi$ denote $C \subseteq^* \{n : \varphi(n)\da = \psi(n)\da\}$.  The relation $=_C$ is an equivalence relation on $D$.  Let $[\varphi]$ denote the equivalence class of $\varphi \in D$ with respect to $=_C$.

\medskip

\item The domain of $\prod_C \mc{A}_n$ is the set $|\prod_C \mc{A}_n| = \{[\varphi] : \varphi \in D\}$.

\medskip

\item 
Let $R$ be an $m$-ary relation symbol of $\mf{L}$.  For each $[\varphi_0], \dots, [\varphi_{m-1}] \in |\prod_C \mc{A}_n|$, define $R^{\prod_C \mc{A}_n}([\varphi_0], \dots, [\varphi_{m-1}])$ by
\begin{align*}
R^{\prod_C \mc{A}_n}([\varphi_0], \dots, [\varphi_{m-1}]) \;\;\Biimp\;\; C \subseteq^* \bigl\{n : R^{\mc{A}_n}(\varphi_0(n), \dots, \varphi_{m-1}(n))\bigr\}.
\end{align*}
Here, $R^{\mc{A}_n}(\varphi_0(n), \dots, \varphi_{m-1}(n))$ includes the condition that $\varphi_i(n)\da$ for each $i < m$.

\medskip

\item Let $f$ be an $m$-ary function symbol of $\mf{L}$.  For each $[\varphi_0], \dots, [\varphi_{m-1}] \in |\prod_C \mc{A}_n|$, let $\psi$ be the partial computable function defined by
\begin{align*}
\psi(n) \keq f^{\mc{A}_n}(\varphi_0(n), \dots, \varphi_{m-1}(n)),
\end{align*}
and notice that $C \subseteq^* \dom(\psi)$ because $C \subseteq^* \dom(\varphi_i)$ for each $i < m$.  Define $f^{\prod_C \mc{A}_n}$ by
\begin{align*}
f^{\prod_C \mc{A}_n}([\varphi_0], \dots, [\varphi_{m-1}]) = [\psi].
\end{align*}

\medskip

\item Let $c$ be a constant symbol of $\mf{L}$.  Let $\psi$ be the total computable function defined by $\psi(n) = c^{\mc{A}_n}$, and define $c^{\prod_C \mc{A}_n} = [\psi]$.
\end{itemize}

In the case where $\mc{A}_n$ is the same fixed computable structure $\mc{A}$ for every $n$, the cohesive product $\prod_C \mc{A}_n$ is called the \emph{cohesive power of $\mc{A}$ over $C$} and is denoted $\prod_C \mc{A}$.
\end{Definition}

As in the classical case, a computable structure $\mc{A}$ always naturally embeds into its cohesive powers.  For $a \in |\mc{A}|$, let $f_a$ be the total computable function with constant value $a$.  Then for any cohesive set $C$, the so-called \emph{canonical embedding} $a \mapsto [f_a]$ embeds $\mc{A}$ into $\prod_C \mc{A}$.

The extent to which analogs of {\L}o\'{s}'s theorem hold for cohesive products and powers depends on what formulas are decidable in the structures.  We make the following definitions as in~\cite{CohPowJournal}.  Say that a computable sequence of formulas $(\Phi_i : i \in \Nb)$ is \emph{uniformly decidable} in a uniformly computable sequence of structures $(\mc{A}_n : n \in \Nb)$ if there is an algorithm that, given an $n$, an $i$, a subformula $\Psi(\vec{y})$ of $\Phi_i$, and a sequence of parameters $\vec{a}$ from $|\mc{A}_n|$ of appropriate length, determines whether $\mc{A}_n \models \Psi(\vec{a})$.  That is, $(\Phi_i : i \in \Nb)$ is uniformly decidable in $(\mc{A}_n : n \in \Nb)$ if the set
\begin{align*}
\{\la n, i, \Psi(\vec{y}), \vec{a} \ra : \text{$\Psi$ is a subformula of $\Phi_i$} \;\land\; |\vec{a}| = |\vec{y}| \;\land\; \mc{A}_n \models \Psi(\vec{a})\}
\end{align*}
is computable.  In the case of a single formula $\Phi$ (and its subformulas), we say that $\Phi$ is \emph{uniformly decidable} in $(\mc{A}_n : n \in \Nb)$.

\begin{Lemma}[\cite{CohPowJournal}*{Lemma~2.5}]\label{lem-LosProdParamHelper}
Let $\mf{L}$ be a computable language, let $(\mc{A}_n : n \in \Nb)$ be a uniformly computable sequence of $\mf{L}$-structures, and let $C$ be a cohesive set.  Let $\Phi(\vec{x}, \vec{y}, v_0, \dots, v_{m-1})$ be a formula that is uniformly decidable in $(\mc{A}_n : n \in \Nb)$.

\begin{enumerate}[(1)]
\item\label{it-LosProdPramSig2Helper} For any $[\varphi_0], \dots, [\varphi_{m-1}] \in |\prod_C \mc{A}_n|$,
\begin{align*}
\prod\nolimits_C \mc{A}_n \models \exists \vec{x}\, \forall \vec{y}\, \Phi(\vec{x}, \vec{y}, [\varphi_0], \dots, [\varphi_{m-1}]) \;\;\Imp\;\; C \subseteq^* \bigl\{n : \mc{A}_n \models \exists \vec{x}\, \forall \vec{y}\, \Phi(\vec{x}, \vec{y}, \varphi_0(n), \dots, \varphi_{m-1}(n))\bigr\}.
\end{align*}

\medskip

\item\label{it-LosProdPramPi2Helper} For any $[\varphi_0], \dots, [\varphi_{m-1}] \in |\prod_C \mc{A}_n|$,
\begin{align*}
C \subseteq^* \bigl\{n : \mc{A}_n \models \forall \vec{x}\, \exists \vec{y}\, \Phi(\vec{x}, \vec{y}, \varphi_0(n), \dots, \varphi_{m-1}(n))\bigr\} \;\;\Imp\;\; \prod\nolimits_C \mc{A}_n \models \forall \vec{x}\, \exists \vec{y}\, \Phi(\vec{x}, \vec{y}, [\varphi_0], \dots, [\varphi_{m-1}]).
\end{align*}
\end{enumerate}
\end{Lemma}

We obtain the following analogs of {\L}o\'{s}'s theorem from Lemma~\ref{lem-LosProdParamHelper}.  We abuse the terminology somewhat by saying that a formula is $\Delta_n$ if it is logically equivalent to both a $\Sigma_n$ formula and a $\Pi_n$ formula.

\begin{Theorem}[\cite{CohPowJournal}*{Theorem~2.7}]\label{thm-LosProdParam}
Let $\mf{L}$ be a computable language, let $(\mc{A}_i : i \in \Nb)$ be a uniformly $n$-decidable sequence of $\mf{L}$-structures, and let $C$ be a cohesive set.
\begin{enumerate}[(1)]
\item\label{it-LosProdPramSig2} Let $\Phi(v_0, \dots, v_{m-1})$ be a $\Sigma_{n+2}$ formula.  Then for any $[\varphi_0], \dots, [\varphi_{m-1}] \in |\prod_C \mc{A}_i|$,
\begin{align*}
\prod\nolimits_C \mc{A}_i \models \Phi([\varphi_0], \dots, [\varphi_{m-1}]) \;\;\Imp\;\; C \subseteq^* \bigl\{i : \mc{A}_i \models \Phi(\varphi_0(i), \dots, \varphi_{m-1}(i))\bigr\}.
\end{align*}

\medskip

\item\label{it-LosProdPramPi2} Let $\Phi(v_0, \dots, v_{m-1})$ be a $\Pi_{n+2}$ formula.  Then for any $[\varphi_0], \dots, [\varphi_{m-1}] \in |\prod_C \mc{A}_i|$,
\begin{align*}
C \subseteq^* \bigl\{i : \mc{A}_i \models \Phi(\varphi_0(i), \dots, \varphi_{m-1}(i))\bigr\} \;\;\Imp\;\; \prod\nolimits_C \mc{A}_i \models \Phi([\varphi_0], \dots, [\varphi_{m-1}]).
\end{align*}

\medskip

\item\label{it-LosProdPramDelta2} Let $\Phi(v_0, \dots, v_{m-1})$ be a $\Delta_{n+2}$ formula.  Then for any $[\varphi_0], \dots, [\varphi_{m-1}] \in |\prod_C \mc{A}_i|$,
\begin{align*}
\prod\nolimits_C \mc{A}_i \models \Phi([\varphi_0], \dots, [\varphi_{m-1}]) \;\;\Biimp\;\; C \subseteq^* \bigl\{i : \mc{A}_i \models \Phi(\varphi_0(i), \dots, \varphi_{m-1}(i))\bigr\}.
\end{align*}
\end{enumerate}
\end{Theorem}

The analog of {\L}o\'{s}'s theorem achieves an extra quantifier in the case of sentences and cohesive powers.

\begin{Theorem}[\cite{CohPowJournal}*{Theorem~2.9}]\label{thm-LosGeneral}
Let $\mf{L}$ be a computable language, let $\mc{A}$ be an $n$-decidable $\mf{L}$-structure, and let $C$ be a cohesive set.

\begin{enumerate}[(1)]
\item\label{it-LosDelta3Sent} Let $\Phi$ be a $\Delta_{n+3}$ sentence.  Then $\mc{A} \models \Phi$ if and only if $\prod_C \mc{A} \models \Phi$.

\medskip

\item\label{it-LosSigma3Sent} Let $\Phi$ be a $\Sigma_{n+3}$ sentence.  If $\mc{A} \models \Phi$, then $\prod_C \mc{A} \models \Phi$.
\end{enumerate}
\end{Theorem}

It follows that the full analogs of {\L}o\'{s}'s theorem hold when the structures are uniformly decidable.  If $(\mc{A}_i : i \in \Nb)$ is a uniformly decidable sequence of structures, then the conclusion of Theorem~\ref{thm-LosProdParam} item~\ref{it-LosProdPramDelta2} holds for every first-order formula $\Phi$.  Similarly, if $\mc{A}$ is a decidable structure, then the conclusion of Theorem~\ref{thm-LosGeneral} item~\ref{it-LosDelta3Sent} holds for every first-order sentence $\Phi$.

In~\cite{CohPowJournal} it is shown (following~\cite{Nelson}) that cohesive products of uniformly decidable sequences of structures are recursively saturated and that, for $n > 0$, cohesive products of uniformly $n$-decidable sequences of structures are $\Sigma_n$-recursively saturated.  More interestingly, it is shown that we obtain an extra level of saturation as well as the $n=0$ case when the cohesive set is assumed to be $\Pi_1$:  cohesive products of uniformly $n$-decidable sequences of structures over $\Pi_1$ cohesive sets are $\Sigma_{n+1}$-recursively saturated~\cite{CohPowJournal}*{Theorem~2.18}.  We now show that the same holds for cohesive products over $\Delta_2$ cohesive sets:  cohesive products of uniformly $n$-decidable sequences of structures over $\Delta_2$ cohesive sets are $\Sigma_{n+1}$-recursively saturated.

Recall the definitions pertaining to saturation.  Let $\mf{L}$ be a language, let $\mc{A}$ be an $\mf{L}$-structure, and let $D \subseteq |\mc{A}|$ be a collection of parameters from $|\mc{A}|$.  Let $\mf{L}_D = \mf{L} \cup D$ be the language obtained by augmenting $\mf{L}$ with fresh constant symbols identified with the members of $D$.  A \emph{type (of $\mc{A}$) over $D$} is a set of $\mf{L}_D$-formulas $p(\vec{x}) = p(x_0, \dots, x_{m-1})$ in $m$ fixed variables $x_0, \dots, x_{m-1}$ that is finitely satisfied in $\mc{A}$:  for every $\Phi_0(\vec{x}), \dots, \Phi_{k-1}(\vec{x}) \in p(\vec{x})$, $\mc{A} \models \exists \vec{x} \, \bigwedge_{i < k} \Phi_i(\vec{x})$.  A type $p(\vec{x})$ of $\mc{A}$ over $D$ is \emph{realized} if there are $a_0, \dots, a_{m-1} \in |\mc{A}|$ such that for all $\Phi(\vec{x}) \in p(\vec{x})$, $\mc{A} \models \Phi(\vec{a})$.  A type is a \emph{$\Sigma_n$-type} if every formula in the type is $\Sigma_n$.  Now let $\mf{L}$ be a computable language.  An $\mf{L}$-structure $\mc{A}$ is \emph{recursively saturated} if it realizes every computable type over a finite set of parameters, and it is \emph{$\Sigma_n$-recursively saturated} if it realizes every computable $\Sigma_n$-type over a finite set of parameters.

Let $p(\vec{x})$ be a type of some structure $\mc{A}$ over parameters $D$.  For us, $D$ is always finite or countable, and we enumerate it as a sequence $\vec{c}$ of appropriate length.  We write $p(\vec{x}; \vec{c})$ for $p(\vec{x})$ and $\Phi(\vec{x}; \vec{c})$ for a formula of $p(\vec{x}; \vec{c})$ when we want to highlight the parameters.  Here, $\Phi(\vec{x}; \vec{c})$ is shorthand for $\Phi(\vec{x}; \vec{c} \rst k)$, where $\vec{c} \rst k$ is the shortest initial segment of $\vec{c}$ containing all the parameters appearing in $\Phi$.  We also write $\Phi(\vec{x}; \vec{y})$ for the $\mc{L}$-formula corresponding to $\Phi(\vec{x}; \vec{c})$, with fresh variables $\vec{y}$ in place of the parameters $\vec{c}$.

The following lemma extends~\cite{CohPowJournal}*{Lemma~2.17} by allowing $\Delta_2$ cohesive sets and by allowing infinite sequences of uniformly partial computable parameters.  Say that a sequence $([\psi_\ell] : \ell \in \Nb)$ of elements of some cohesive product is \emph{uniformly partial computable} if the sequence $(\psi_\ell : \ell \in \Nb)$ of representatives is uniformly partial computable.

\begin{Lemma}[Extending~\cite{CohPowJournal}*{Lemma~2.17}]\label{lem-SatDelta2}
Let $\mf{L}$ be a computable language, let $(\mc{A}_n : n \in \Nb)$ be a uniformly computable sequence of $\mf{L}$-structures, and let $C$ be a $\Delta_2$ cohesive set.  Let $p  = p(\vec{x}; \ora{[\psi]})$ be a computable type of $\prod_C \mc{A}_n$ over a uniformly partial computable sequence of parameters $\ora{[\psi]} = ([\psi_\ell] : \ell \in \Nb)$.  Assume that $p$ consists of formulas of the form $\exists \vec{z}\, \Phi(\vec{x}, \vec{z}; \ora{[\psi]})$ with computable enumeration $(\exists \vec{z}_i \, \Phi_i(\vec{x}, \vec{z}_i; \ora{[\psi]}) : i \in \Nb)$.  Further assume that the formulas $(\Phi_i(\vec{x}, \vec{z}_i; \vec{y}) : i \in \Nb)$ are uniformly decidable in the structures $(\mc{A}_n : n \in \Nb)$.  Then $\prod_C \mc{A}_n$ realizes $p$.
\end{Lemma}

\begin{proof}
As $p(\vec{x}; \ora{[\psi]})$ is a type,
\begin{align*}
\prod\nolimits_C \mc{A}_n \models \exists \vec{x}\, \bigwedge_{i < k} \exists \vec{z}_i\, \Phi_i(\vec{x}, \vec{z}_i; \ora{[\psi]}) 
\end{align*}
for each $k$.  To streamline the notation, let $\psi \colon \Nb^2 \to \Nb^{<\Nb}$ be the partial computable function given by $\psi(i,n) \keq \la \psi_0(n), \dots, \psi_{\ell_i-1}(n) \ra$, where $\ell_i$ is least such that $\ora{[\psi]} \rst \ell_i$ contains all the parameters of $\ora{[\psi]}$ appearing in $\Phi_i$.  Notice that $C \subseteq^* \{n : \psi(i,n)\da\}$ for every $i$.

Our goal is to partially compute a function $\theta \colon \Nb \to \Nb^m$ so that $C \subseteq^* \dom(\theta)$ and
\begin{align*}
\existsinf n \in C \; \Bigl( \mc{A}_n \models \exists \vec{z}_i\, \Phi_i(\theta(n), \vec{z}_i; \psi(i,n)) \Bigr) \tag{$*$}\label{eq-InfSat}
\end{align*}
for each $i$.  The set
\begin{align*}
\Bigl\{n : \mc{A}_n \models \exists \vec{z}_i\, \Phi_i(\theta(n), \vec{z}_i; \psi(i,n))\Bigr\}
\end{align*}
is c.e.\ for each $i$ because $\Phi_i$ is uniformly decidable in $(\mc{A}_n : n \in \Nb)$.  Thus \eqref{eq-InfSat} implies that 
\begin{align*}
\forae n \in C \; \Bigl( \mc{A}_n \models \exists \vec{z}_i\, \Phi_i(\theta(n), \vec{z}_i; \psi(i,n)) \Bigr)
\end{align*}
for each $i$ by cohesiveness.   Once $\theta$ has been defined, we let $\varphi_j = \pi_j \circ \theta$ for each $j < m$.  We then have that $[\varphi_0], \dots, [\varphi_{m-1}] \in |\prod_C \mc{A}_n|$ and that
\begin{align*}
\prod\nolimits_C \mc{A}_n \models \exists \vec{z}_i\, \Phi_i\left(\ora{[\varphi]}, \vec{z}_i; \ora{[\psi]} \right)
\end{align*}
for each $i$ by Lemma~\ref{lem-LosProdParamHelper} item~\ref{it-LosProdPramPi2Helper}.  Thus $[\varphi], \dots, [\varphi_{m-1}]$ realize $p(\vec{x}; \ora{[\psi]})$ in $\prod_C \mc{A}_n$.

Let $f \colon \Nb^2 \to \{0,1\}$ be a $\Delta_2$-approximation to the cohesive set $C$.  Let $(U_k : k \in \Nb)$ be the uniformly c.e.\ sequence of sets given by
\begin{align*}
U_k = \left\{\la \vec{a}, n \ra \in \Nb^m \times \Nb: \mc{A}_n \models \bigwedge_{i < k} \exists \vec{z}_i\, \Phi_i(\vec{a}, \vec{z}_i; \psi(i,n))\right\}
\end{align*}
with uniformly computable $\subseteq$-increasing enumerations $(U_{k,s})_{s \in \Nb}$ of finite sets for each $k$.  These enumerations are given in terms of strong indices, so membership in $U_{k,s}$ as well as its size are uniformly computable in $k$ and $s$.  The sequence $(U_k : k \in \Nb)$ is uniformly c.e.\ because the formulas $(\Phi_i : i \in \Nb)$ are uniformly decidable in $(\mc{A}_n : n \in \Nb)$.  Notice that if $k \leq j$, then $U_j \subseteq U_k$.  We can therefore arrange the enumerations so that if $k \leq j$, then $\forall s \, (U_{j, s} \subseteq U_{k, s})$.

We partially compute $\theta$ by computing an increasing sequence $\theta_0 \subseteq \theta_1 \subseteq \theta_2 \subseteq \cdots$ of finite approximations to $\theta$.  Start at stage $0$ with $\theta_0 = \emptyset$.  At stage $s$, we have $\theta_s$, and we define $\theta_{s+1}$.

Say that \emph{$n$ covers $k$ at stage $s$} if the following conditions hold.
\begin{enumerate}[(a)]
\item $n > k$.

\medskip

\item\label{it-SatCondInC} $f(n,s) = 1$.

\medskip

\item $\theta_s(n)\da$.

\medskip

\item $\la \theta_s(n), n \ra \in U_{k,s}$.
\end{enumerate}

If there is an $n$ that covers $k$ at stage $s$, then we say that $k$ is \emph{covered} at stage $s$.  Let $k^0_s$ be the least number that is not covered at the start of stage $s$.  If $s > 0$, then let $k^1_s$ be the least number (if it exists) for which there is an $n$ with $f(n,s) = 0$ that covered $k^1_s$ at stage $s-1$, but no $\wh{n} \leq n$ covers $k^1_s$ at stage $s$.  If $k^1_s$ is defined, then let $k_s = \min\{k^0_s, k^1_s\}$.  Otherwise, let $k_s = k^0_s$.  Now check if there is an $n < s$ meeting the following conditions.
\begin{enumerate}[(i)]
\item\label{it-nGreaterKs} $n > k_s$.

\medskip

\item $f(n,s) = 1$.

\medskip

\item $\theta_s(n)\ua$.

\medskip

\item\label{it-WitnessVec} There is an $\vec{a}$ with $\la \vec{a}, n \ra \in U_{k_s, s}$.
\end{enumerate}
If there is such an $n$, choose the least such $n$, choose the first corresponding $\vec{a}$ as in item~\ref{it-WitnessVec}, and extend $\theta_s$ to $\theta_{s+1}$ by setting $\theta_{s+1}(n) = \vec{a}$.  If there is no such $n$, then set $\theta_{s+1} = \theta_s$.  Now go to stage $s+1$.  This completes the definition of $\theta$.

Suppose that $n$ covers $k$ at some stage $s_0$.  Then the only way that $n$ could fail to cover $k$ at some stage $s > s_0$ is by the failure of condition~\ref{it-SatCondInC}.  If, however, $n \in C$, then there is a stage $s_1 > s_0$ such that $\forall s \geq s_1 \; (f(n,s) = 1)$.  Then $n$ covers $k$ at all stages $s \geq s_1$.  In this situation, we say that \emph{$k$ is covered by an $n \in C$}.  We show by induction on $k$ that every $k$ is eventually covered by an $n \in C$.

\begin{Claim*}
Every $k$ is eventually covered by an $n \in C$.
\end{Claim*}

\begin{proof}[Proof of Claim]
Proceed by induction on $k$.  Let $s_0$ be a stage by which all $\wh{k} < k$ have been covered by members of $C$.  Let $c$ be the greatest member of $C$ covering a $\wh{k} < k$ at stage $s_0$.  Let $s_1 > s_0$ be a stage by which $f$ has settled to its final value on all $n$ up to $c$:  $\forall n \leq c \; \forall s \geq s_1 \; (f(n,s) = C(n))$.  Then $k_s \geq k$ at all stages $s > s_1$.  By assumption,
\begin{align*}
\prod\nolimits_C \mc{A}_n \models \exists \vec{x}\, \bigwedge_{i < k} \exists \vec{z}_i\, \Phi_i\left(\vec{x}, \vec{z}_i; \ora{[\psi]} \right),
\end{align*}
and therefore
\begin{align*}
C \subseteq^* \left\{n : \mc{A}_n \models \exists \vec{x}\, \bigwedge_{i < k} \exists \vec{z}_i\, \Phi_i\left(\vec{x}, \vec{z}_i; \psi(i,n) \right)\right\}
\end{align*}
by Lemma~\ref{lem-LosProdParamHelper} item~\ref{it-LosProdPramSig2Helper}.  The lemma applies because prenexing the formula $\exists \vec{x}\, \bigwedge_{i < k} \exists \vec{z}_i\, \Phi_i(\vec{x}, \vec{z}_i; \vec{y})$ yields a formula of the form $\exists \vec{w}\, \Psi(\vec{w}; \vec{y})$, where $\Psi$ is uniformly decidable in $(\mc{A}_n : n \in \Nb)$.  Let $n_0$ be least with $n_0 > k$, $n_0 \in C$, $\theta_{s_1}(n_0)\ua$, and $\exists \vec{a} \, (\la \vec{a}, n_0 \ra \in U_k)$.  If $\theta_s(n_0)$ is defined for the first time during a stage $s > s_1$, then it is to cover a $j$ with $k \leq j < n_0$.  That is, if the value of $\theta_s(n_0)$ is determined at stage $s > s_1$, then it is chosen so that $\la \theta_s(n_0), n_0 \ra \in U_{j, s} \subseteq U_{k, s}$ for a $j \geq k$.  Therefore $n_0$ covers $k$ at any stage $s > s_1$ at which $\theta_s(n_0)\da$ and $f(n_0, s) = 1$.

Let $s_2 > \max\{n_0, s_1\}$ be large enough so that $\forall n \leq n_0 \; \forall s \geq s_2 \; (f(n,s) = C(n))$ and that $\exists \vec{a} \, (\la \vec{a}, n_0 \ra \in U_{k, s_2})$.  Consider stage $s_2$.  If $k$ is not covered at stage $s_2$, then it must be that $k_{s_2} = k$ and that $\theta_{s_2}(n_0)\ua$.  Furthermore, by choice of $n_0$ and $s_2$, $n_0 < s_2$ is the least number meeting conditions~\ref{it-nGreaterKs}--\ref{it-WitnessVec} at stage $s_2$.  Therefore $\theta_{s_2 + 1}(n_0)$ is defined to cover $k$ at stage $s_2$, so $k$ is covered by an element of $C$.

Suppose instead that $k$ is covered at stage $s_2$.  Let $n$ be the least number for which there is a stage $s_3 \geq s_2$ at which $n$ covers $k$.  If $n \in C$, then $k$ is covered by an element of $C$, as desired.  If $n \notin C$, then there is a least stage $s > s_3$ with $f(n,s) = 0$.  The number $n$ covers $k$ at stage $s-1$, but by the choice of $n$, no $\wh{n} \leq n$ covers $k$ at stage $s$.  Thus $k^1_s = k$, so $k_s = k$.  If $\theta_s(n_0)\da$, then $n_0$ must already cover $k$ as observed above.  If $\theta_s(n_0)\ua$, then $n_0 < s$ is the least number meeting conditions~\ref{it-nGreaterKs}--\ref{it-WitnessVec} at stage $s$.  Therefore $\theta_{s + 1}(n_0)$ is defined to cover $k$ at stage $s$, so $k$ is covered by an element of $C$.  This completes the proof of the claim.
\end{proof}

To complete the proof, consider the formula $\exists \vec{z}_i\, \Phi_i$.  By the claim, every $k$ is eventually covered by an $n \in C$.  Thus for every $k > i$, there is an $n > k$ with $n \in C$, $\theta(n)\da$, and $\la \theta(n), n \ra \in U_k$.  Thus $C \subseteq^* \dom(\theta)$ by cohesiveness, and
\begin{align*}
\existsinf n \in C \; \Bigl( \mc{A}_n \models \exists \vec{z}_i\, \Phi_i(\theta(n), \vec{z}_i; \psi(i,n)) \Bigr)
\end{align*}
as desired.
\end{proof}

\begin{Theorem}[Extending~\cite{CohPowJournal}*{Theorem~2.18}]\label{thm-SatDelta2}
Let $\mf{L}$ be a computable language, and let $C$ be a $\Delta_2$ cohesive set.
\begin{enumerate}[(1)]
\item\label{it-nDecRecSatSeqDelta2} Let $(\mc{A}_i : i \in \Nb)$ be a uniformly $n$-decidable sequence of $\mf{L}$-structures.  Then $\prod_C \mc{A}_i$ realizes every computable $\Sigma_{n+1}$-type over a uniformly partial computable sequence of parameters.  In particular, $\prod_C \mc{A}_i$ is $\Sigma_{n+1}$-recursively saturated.

\medskip

\item\label{it-nDecRecSatStrDelta2} Let $\mc{A}$ be an $n$-decidable $\mf{L}$-structure.  Then $\prod_C \mc{A}$ realizes every computable $\Sigma_{n+1}$-type over a uniformly partial computable sequence of parameters.  In particular, $\prod_C \mc{A}$ is $\Sigma_{n+1}$-recursively saturated.
\end{enumerate}
\end{Theorem}

\begin{proof}
Item~\ref{it-nDecRecSatSeqDelta2} follows from Lemma~\ref{lem-SatDelta2}.  A computable $\Sigma_{n+1}$-type can be computably enumerated as $(\exists \vec{z}_j\, \Phi_j : j \in \Nb)$, where $\Phi_j$ is $\Pi_n$ for every $j$.  The formulas $(\Phi_j : j \in \Nb)$ are then uniformly decidable in the uniformly $n$-decidable sequence of structures $(\mc{A}_i : i \in \Nb)$.  Item~\ref{it-nDecRecSatStrDelta2} is the special case of item~\ref{it-nDecRecSatSeqDelta2} where $\mc{A}_i$ is $\mc{A}$ for each $i$.
\end{proof}

\subsection*{Cohesive products and powers of computable linear orders}

A \emph{linear order} $\mc{L} = (L; \prec)$ consists of a non-empty set $L$ equipped with a binary relation $\prec$ satisfying the following axioms.
\begin{itemize}
\item $\forall x \, (x \nprec x)$.

\smallskip

\item $\forall x \, \forall y \, \forall z \, \bigl( (x \prec y \,\land\, y \prec z) \;\imp\; x \prec z \bigr)$.

\smallskip

\item $\forall x \, \forall y \, (x \prec y \,\lor\, x = y \,\lor\, y \prec x)$.
\end{itemize}
Furthermore, a linear order $\mc{L}$ is \emph{dense} if $\forall x \, \forall y \, \exists z \, (x \prec y \;\imp\; x \prec z \prec y)$ and \emph{has no endpoints} if $\forall x \, \exists y \, \exists z \, (y \prec x \prec z)$.  A computable linear order $\mc{L} = (L; \prec_\mc{L})$ therefore consists of a non-empty computable set $L \subseteq \Nb$ and a computable relation ${\prec_\mc{L}} \,\subseteq\, L \times L$ that linearly orders $L$.

Let $\mc{L} = (L; \prec_\mc{L})$ be a linear order, and let $a, b \in L$.  Let $\min_{\prec_\mc{L}} \{a, b\}$ and $\max_{\prec_\mc{L}} \{a, b\}$ denote the minimum and maximum of $a$ and $b$ with respect to $\prec_\mc{L}$; let $(a,b)_\mc{L} = \{x \in L : a \prec_\mc{L} x \prec_\mc{L} b\}$ and $[a,b]_\mc{L} = \{x \in L : a \preceq_\mc{L} x \preceq_\mc{L} b\}$ denote the open and closed intervals defined by $a$ and $b$; and let $|(a,b)_\mc{L}|$ and $|[a,b]_\mc{L}|$ denote the cardinalities of the respective intervals.  It is convenient to allow $b \prec_\mc{L} a$ in the interval notation, in which case $(a,b)_\mc{L} = [a,b]_\mc{L} = \emptyset$.  Let $a \pprec_\mc{L} b$ denote that the interval $(a,b)_\mc{L}$ is infinite.

Recall that $\omega$ denotes the order-type of the natural numbers $(\Nb; <)$, that $\zeta$ denotes the order-type of the integers $(\Zb; <)$, and that $\eta$ denotes the order-type of the rationals $(\Qb; <)$, all with their usual orders.  We refer to $(\Nb; <)$, $(\Zb; <)$, and $(\Qb; <)$ as the \emph{usual presentations} of $\omega$, $\zeta$, and $\eta$.  For each $n \geq 1$, let $\bm{n}$ denote the order-type of the finite linear order $(\{0, 1, \dots, n-1\}; <)$.  For any order-type $\alpha$, a computable linear order of type $\alpha$ is called a \emph{computable copy} of $\alpha$.  A computable copy of $\omega$ is not necessarily isomorphic to the usual presentation of $\omega$ via a computable isomorphism.  Every countable dense linear order without endpoints has order-type $\eta$, and every computable copy of $\eta$ is computably isomorphic to the usual presentation of $\eta$ (see~\cite{RosBook}*{Theorem~2.8 and Exercise~16.4}).

Let $\mc{L}_0 + \mc{L}_1$ and $\mc{L}_0\mc{L}_1$ denote the usual sum and product of linear orders $\mc{L}_0$ and $\mc{L}_1$.  Let $\mc{L}^*$ denote the reverse of the linear order $\mc{L}$.  Furthermore, recall the \emph{generalized sum} of a sequence $(\mc{M}_\ell : \ell \in |\mc{L}|)$ of linear orders indexed by the elements of a linear order $\mc{L}$.

\begin{Definition}[see~\cite{RosBook}*{Definition~1.38}]\label{def-GenSum}
Let $\mc{L}$ be a linear order, and let $(\mc{M}_\ell : \ell \in |\mc{L}|)$ be a sequence of linear orders indexed by $|\mc{L}|$.  The \emph{generalized sum} $\sum_{\ell \in |\mc{L}|}\mc{M}_\ell$ of $(\mc{M}_\ell : \ell \in |\mc{L}|)$ over $\mc{L}$ is the linear order $\mc{S} = (S; \prec_\mc{S})$ defined as follows.  Write $\mc{L} = (L; \prec_\mc{L})$, and write $\mc{M}_\ell = (M_\ell; \prec_{\mc{M}_\ell})$ for each $\ell \in L$.  Define $S = \{(\ell, m) : \ell \in L \,\land\, m \in M_\ell\}$, and define
\begin{align*}
(\ell_0, m_0) \prec_{\mc{S}} (\ell_1, m_1) \quad\text{if and only if}\quad (\ell_0 \prec_\mc{L} \ell_1) \;\lor\; (\ell_0 = \ell_1 \,\land\, m_0 \prec_{\mc{M}_{\ell_0}} m_1).
\end{align*}
\end{Definition}

Generalized sums may be used to define the \emph{shuffle} of a non-empty and at-most-countable collection of linear orders.

\begin{Definition}[see~\cite{RosBook}*{Definition~7.14}]\label{def-Shuffle}
Let $X$ be a non-empty collection of linear orders with $|X| \leq \aleph_0$.  Let $f \colon \Qb \to X$ be a function such that $f^{-1}(\mc{M})$ is dense in $\Qb$ for each linear order $\mc{M} \in X$.  Let $\mc{S} = \sum_{q \in \Qb} f(q)$ be the generalized sum of the sequence $(f(q) : q \in \Qb)$ over $\Qb$.  By density, the order-type of $\mc{S}$ does not depend on the particular choice of $f$.  Therefore $\mc{S}$ is called the \emph{shuffle} of $X$ and is denoted $\bm{\sigma}(X)$.
\end{Definition}

In Definition~\ref{def-Shuffle}, it is helpful to identify each linear order in $X$ with a unique color and think of the function $f \colon \Qb \to X$ as a coloring of $\Qb$ in which every color occurs densely.  The shuffle $\bm{\sigma}(X)$ is then obtained by replacing each element of $\Qb$ by the linear order with which it is colored.  Also, we usually think of the $X$ in a shuffle $\bm{\sigma}(X)$ as a collection of order-types instead of as a collection of concrete linear orders.

If $\mc{L}$ is a computable linear order and $(\mc{M}_\ell : \ell \in |\mc{L}|)$ is a uniformly computable sequence of linear orders indexed by $|\mc{L}|$, then the pairing function may be used to compute a copy of $\sum_{\ell \in |\mc{L}|}\mc{M}_\ell$.  Furthermore, if $(\mc{M}_n : n \in \Nb)$ is a uniformly computable sequence of linear orders, then a computable dense coloring $f \colon \Qb \to \Nb$ of $\Qb$ may be used to compute a copy of $\bm{\sigma}(\{\mc{M}_n : n \in \Nb\})$.

Recall now the \emph{condensations} and in particular the \emph{finite condensation} of a linear order.  In general, a condensation of a linear order $\mc{L}$ is obtained by partitioning $\mc{L}$ into non-empty intervals and then by collapsing each interval to a point.

\begin{Definition}\label{def-Cond}
Let $\mc{L} = (L; \prec_\mc{L})$ be a linear order.  A \emph{condensation} of $\mc{L}$ is any linear order $\mc{M} = (M; \prec_\mc{M})$ obtained by partitioning $L$ into a collection $M$ of non-empty intervals and, for intervals $I, J \in M$, defining $I \prec_\mc{M} J$ if and only if $\forall a \in I \; \forall b \in J \; (a \prec_\mc{L} b)$.
\end{Definition}

\begin{Definition}\label{def-FinCond}
Let $\mc{L} = (L; \prec_\mc{L})$ be a linear order.  For $x \in L$, let $\condF(x)$ denote the set of $y \in L$ for which there are only finitely many elements between $x$ and $y$:
\begin{align*}
\condF(x) = \Bigl\{y \in L : \text{the interval $\bigl[\min\nolimits_{\prec_\mc{L}}\{x,y\}, \max\nolimits_{\prec_\mc{L}}\{x,y\}\bigr]_{\mc{L}}$ is finite} \Bigl\}.
\end{align*}
The set $\condF(x)$ is always a non-empty interval, as $x \in \condF(x)$.  The \emph{finite condensation} $\condF(\mc{L})$ of $\mc{L}$ is the condensation obtained from the partition $\{\condF(x) : x \in L\}$.
\end{Definition}

When a linear order $\mc{L} = (L; \prec_\mc{L})$ is partitioned into non-empty intervals as in Definition~\ref{def-Cond}, the intervals of the partition are called the \emph{blocks} of $\mc{L}$.  For the finite condensation of $\mc{L}$, the blocks are the sets of the form $\condF(x)$ for $x \in L$, each of which has order-type either $\omega$, $\omega^*$, $\zeta$, or $\bm{n}$ for some $n \geq 1$.  For $x, y \in L$, we also have that $\condF(x) \prec_{\condF(\mc{L})} \condF(y)$ if and only if $x \pprec_\mc{L} y$.

Linear orders are axiomatized by $\Pi_1$ sentences, so it follows from Theorem~\ref{thm-LosProdParam} item~\ref{it-LosProdPramPi2} that a cohesive product of a uniformly computable sequence of linear orders is again a linear order.  The properties of a linear order being dense and having no endpoints can each be expressed by $\Pi_2$ sentences, so also a cohesive product of a uniformly computable sequence of dense linear orders is a dense linear order, and a cohesive product of a uniformly computable sequence of linear orders without endpoints is a linear order without endpoints.

In~\cite{CohPowJournal} it is shown that cohesive powers commute with the sum, product, and reverse operations on computable linear orders.

\begin{Theorem}[\cite{CohPowJournal}*{Theorem~3.6}]\label{thm-CohPres}
Let $\mc{L}_0$ and $\mc{L}_1$ be computable linear orders, and let $C$ be a cohesive set.  Then
\begin{enumerate}[(1)]
\item\label{it-SumIso} $\prod_C(\mc{L}_0 + \mc{L}_1) \iso \prod_C\mc{L}_0 + \prod_C\mc{L}_1$,

\smallskip

\item\label{it-ProdIso} $\prod_C(\mc{L}_0\mc{L}_1) \iso \bigl( \prod_C\mc{L}_0 \bigr) \bigl( \prod_C\mc{L}_1\bigr)$, and

\smallskip

\item\label{it-RevIso} $\prod_C(\mc{L}_0^*) \iso \bigl( \prod_C\mc{L}_0 \bigr)^*$.
\end{enumerate}
\end{Theorem}

In~\cite{CohPowJournal}, it is also shown that the finite condensation of the cohesive product of a uniformly computable sequence of linear orders over a $\Pi_1$ cohesive set is always dense.  Furthermore, it is shown that the finite condensation of the cohesive power of a computable copy of $\omega$ over a $\Pi_1$ cohesive set always has order-type $\bm{1} + \eta$.  We extend both results to $\Delta_2$ cohesive sets and recall a few helpful lemmas from~\cite{CohPowJournal}.

\begin{Lemma}[\cite{CohPowJournal}*{Lemma~3.7}]\label{lem-ImmedSucc}
Let $(\mc{L}_n : n \in \Nb)$ be a uniformly computable sequence of linear orders, let $C$ be a cohesive set, and let $[\psi]$ and $[\varphi]$ be elements of $\prod_C \mc{L}_n$.  Then the following are equivalent.
\begin{enumerate}[(1)]
\item\label{it-ImmedSuccInPow} $[\varphi]$ is the $\prec_{\prod_C \mc{L}_n}$-immediate successor of $[\psi]$.

\smallskip

\item\label{it-aeImmedSucc} $\forae n \in C \; \bigl(\text{$\varphi(n)$ is the $\prec_{\mc{L}_n}$-immediate successor of $\psi(n)$}\bigr)$.

\smallskip

\item\label{it-existsinfImmedSucc} $\existsinf n \in C \; \bigl(\text{$\varphi(n)$ is the $\prec_{\mc{L}_n}$-immediate successor of $\psi(n)$}\bigr)$.
\end{enumerate}
\end{Lemma}

\begin{Lemma}[\cite{CohPowJournal}*{Lemma~3.8}]\label{lem-DifferentBlocks}
Let $(\mc{L}_n : n \in \Nb)$ be a uniformly computable sequence of linear orders, let $C$ be a cohesive set, and let $[\psi]$ and $[\varphi]$ be elements of $\prod_C \mc{L}_n$.  Then the following are equivalent.
\begin{enumerate}[(1)]
\item\label{it-FarBelow} $[\psi] \pprec_{\prod_C \mc{L}_n} [\varphi]$.

\smallskip

\item\label{it-BelowLim} $\lim_{n \in C}|(\psi(n), \varphi(n))_{\mc{L}_n}| = \infty$.

\smallskip

\item\label{it-BelowLimSup} $\limsup_{n \in C}|(\psi(n), \varphi(n))_{\mc{L}_n}| = \infty$.
\end{enumerate}
\end{Lemma}

\begin{Theorem}[Extending~\cite{CohPowJournal}*{Theorem~3.9}]\label{thm-DenseBlocks}
Let $(\mc{L}_n : n \in \Nb)$ be a uniformly computable sequence of linear orders, and let $C$ be a $\Delta_2$ cohesive set.  Then $\condF(\prod_C \mc{L}_n)$ is dense.
\end{Theorem}

\begin{proof}
The cohesive product $\prod_C \mc{L}_n$ is $\Sigma_1$-recursively saturated by Theorem~\ref{thm-SatDelta2}, so it suffices to show that the finite condensation of a $\Sigma_1$-recursively saturated linear order is dense.

Let $\mc{M} = (M; \prec_\mc{M})$ be a $\Sigma_1$-recursively saturated linear order, and let $a, b \in M$ be such that $a \pprec_\mc{M} b$.  We need to find a $c \in M$ with $a \pprec_\mc{M} c \pprec_\mc{M} b$.  For each $k \in \Nb$, let $\Phi_k(x; a, b)$ be the following formula (with parameters $a$ and $b$) expressing that the intervals $(a, x)_\mc{M}$ and $(x, b)_\mc{M}$ have at least $k$ elements each:
\begin{align*}
\Phi_k(x;a,b) \quad\equiv\quad \exists w_0, \dots, w_{k-1}, z_0, \dots, z_{k-1} \, (a \prec w_0 \prec \cdots \prec w_{k-1} \prec x \prec z_0 \prec \cdots \prec z_{k-1} \prec b).
\end{align*}
The set $p(x;a,b) = \{\Phi_k(x; a, b) : k \in \Nb\}$ is a computable set of $\Sigma_1$ formulas, and it is a type over $\{a,b\}$ because the interval $(a,b)_\mc{M}$ is infinite.  Thus $p(x;a,b)$ is realized by some $c \in M$ by $\Sigma_1$-recursive saturation.  The intervals $(a, c)_\mc{M}$ and $(c, b)_\mc{M}$ are therefore both infinite, so $a \pprec_\mc{M} c \pprec_\mc{M} b$ as desired.
\end{proof}

\begin{Lemma}[\cite{CohPowJournal}*{Lemma~4.1}]\label{lem-StdInitSeg}
Let $\mc{L}$ be a computable copy of $\omega$, and let $C$ be a cohesive set.  Then the image of the canonical embedding of $\mc{L}$ into $\prod_C \mc{L}$ is an initial segment of $\prod_C \mc{L}$ of order-type $\omega$.
\end{Lemma}

Let $\mc{L} = (L; \prec_\mc{L})$ be a computable copy of $\omega$, and let $C$ be a cohesive set.  It is straightforward to check that if $\varphi \colon \Nb \to L$ is a total computable injection, then $[\varphi]$ is not in the image of the canonical embedding.  Therefore the cohesive power $\prod_C \mc{L}$ has the form $\omega + \mc{M}$ for some non-empty linear order $\mc{M}$.  Call an element of $\prod_C \mc{L}$ \emph{standard} if it is in the image of the canonical embedding, and call the element \emph{non-standard} otherwise.

\begin{Lemma}[\cite{CohPowJournal}*{Lemma~4.2}]\label{lem-NonstdUnbdd}
Let $\mc{L}$ be a computable copy of $\omega$, let $C$ be a cohesive set, and let $[\varphi]$ be an element of $\prod_C \mc{L}$.  Then $[\varphi]$ is non-standard if and only if $\liminf_{n \in C} \varphi(n) = \infty$.
\end{Lemma}

In Lemma~\ref{lem-NonstdUnbdd}, the condition $\liminf_{n \in C} \varphi(n) = \infty$ may be replaced by $\lim_{n \in C} \varphi(n) = \infty$ and by $\limsup_{n \in C} \varphi(n) = \infty$.

\begin{Lemma}[\cite{CohPowJournal}*{Lemma~4.3}]\label{lem-NonStdNoEnd}
Let $\mc{L}$ be a computable copy of $\omega$, let $C$ be a cohesive set, and let $[\varphi]$ be a non-standard element of $\prod_C \mc{L}$.  Then there are non-standard elements $[\psi^-]$ and $[\psi^+]$ of $\prod_C \mc{L}$ with $[\psi^-] \pprec_{\prod_C \mc{L}} [\varphi] \pprec_{\prod_C \mc{L}} [\psi^+]$.
\end{Lemma}

\begin{Theorem}[Extending~\cite{CohPowJournal}*{Theorem~4.4}]\label{thm-CoCeDenseCond}
Let $\mc{L}$ be a computable copy of $\omega$, and let $C$ be a $\Delta_2$ cohesive set.  Then $\condF(\prod_C \mc{L})$ has order-type $\bm{1} + \eta$.
\end{Theorem}

\begin{proof}
The standard elements of $\prod_C \mc{L}$ form a single initial block by Lemma~\ref{lem-StdInitSeg}.  The blocks of $\prod_C \mc{L}$ containing the non-standard elements form a countable linear order that is dense by Theorem~\ref{thm-DenseBlocks} and has no endpoints by Lemma~\ref{lem-NonStdNoEnd}.  Thus $\condF(\prod_C \mc{L}) \iso \bm{1} + \eta$.
\end{proof}

The main result of~\cite{CohPowJournal} is that for a variety of countable linear orders $\mc{M}$ with $\condF(\mc{M}) \iso \bm{1} + \eta$ and for a given $\Pi_1$ cohesive set $C$, it is possible to design a computable copy $\mc{L}$ of $\omega$ that achieves $\prod_C \mc{L} \iso \mc{M}$.

\begin{Theorem}[\cite{CohPowJournal}*{Theorem~6.10}]\label{thm-BCSigma2ShuffleOld}
Let $X \subseteq \Nb \setminus \{0\}$ be a Boolean combination of $\Sigma_2$ sets, thought of as a set of finite order-types.  Let $C$ be a $\Pi_1$ cohesive set.  Then there is a computable copy $\mc{L}$ of $\omega$ where the cohesive power $\prod_C \mc{L}$ has order-type $\omega + \bm{\sigma}(X \cup \{\omega + \zeta\eta + \omega^*\})$.  Moreover, if $X$ is finite and non-empty, then there is also a computable copy $\mc{L}$ of $\omega$ where the cohesive power $\prod_C \mc{L}$ has order-type $\omega + \bm{\sigma}(X)$.
\end{Theorem}

The proof of Theorem~\ref{thm-BCSigma2ShuffleOld} involves a coloring apparatus whereby a computable copy of $\omega$ is colored in such a way as to induce a coloring on its cohesive power with certain density properties.  We end this section by introducing the coloring apparatus.

\begin{Definition}[\cite{CohPowJournal}*{Definition~5.1}]
A \emph{colored linear order} is a structure $\mc{O} = (L, \Nb; \prec_\mc{L}, F)$, where $\mc{L} = (L; \prec_\mc{L})$ is a linear order and $F$ is (the graph of) a function $F \colon L \to \Nb$, thought of as a coloring of $L$.  The colored linear order $\mc{O}$ is a \emph{colored copy of $\omega$} if $\mc{L} \iso \omega$.
\end{Definition}
A colored linear order is a two-sorted structure, where one sort is the domain $L$ of the linear order and the other sort is the set $\Nb$ of colors.  The language consists of unary relation symbols for $L$ and $\Nb$, a binary relation symbol for $F$, and the binary relation symbol $\prec$.  Though we must technically formalize $F$ as a relation, we use the function notation $F(\ell) = d$ instead of the relation notation $F(\ell, d)$.

We typically follow the notational convention of letting $\mc{O} = (L, \Nb; \prec_\mc{L}, F)$ denote a colored linear order with the additional coloring structure and letting $\mc{L} = (L; \prec_\mc{L})$ denote the underlying linear order.  So let $\mc{O}$ be a computable colored linear order, let $\mc{L}$ be the underlying linear order, and let $C$ be a cohesive set.  The cohesive power $\prod_C \mc{O}$ consists of a linear order $\prod_C \mc{L}$, a collection of colors $\prod_C \Nb$, and a (graph of a) function $F^{\prod_C \mc{O}} \colon |\prod_C \mc{L}| \to |\prod_C \Nb|$ thought of as a coloring of $\prod_C \mc{L}$.  If $[\varphi]$ is in the domain of $\prod_C \mc{O}$, then, by cohesiveness, either $\varphi(n)$ is an element of the linear order (i.e., is of the first sort) for almost every $n \in C$, or $\varphi(n)$ is a color (i.e., is of the second sort) for almost every $n \in C$.  Therefore $[\varphi]$ is either an element of the linear order $\prod_C \mc{L}$ or an element of the collection of colors $\prod_C \Nb$.  See also the discussion of reducts and substructures of cohesive powers in~\cite{CohPowJournal}*{Section~2}.  That $\prec_\mc{L}$ linearly orders $L$ is expressible by a $\Pi_1$ sentence, so $\prec_{\prod_C \mc{L}}$ indeed linearly orders $|\prod_C \mc{L}|$ by Theorem~\ref{thm-LosGeneral}.  That $F$ is the graph of a function with domain $L$ and codomain $\Nb$ is expressible by a $\Pi_2$ sentence, so $F^{\prod_C \mc{O}}$ is indeed the graph of a function with domain $|\prod_C \mc{L}|$ and codomain $|\prod_C \Nb|$ by Theorem~\ref{thm-LosGeneral} as well.

Again let $\mc{O} = (L, \Nb; \prec_\mc{L}, F)$ be a computable colored linear order, let $\mc{L} = (L; \prec_\mc{L})$, and let $C$ be a cohesive set.  To disambiguate between the elements of the linear order $\prod_C \mc{L}$ and the colors of $\prod_C \Nb$ inside the cohesive power $\prod_C \mc{O}$, we write $[\varphi]$ for the elements of $\prod_C \mc{L}$, and we write $\llb \delta \rrb$ for the elements of $\prod_C \Nb$.  Call a color $\llb \delta \rrb \in |\prod_C \Nb|$ a \emph{solid color} if $\delta$ is eventually constant on $C$:  $\exists d \in \Nb\; \forae n \in C\; (\delta(n) = d)$.  Otherwise, call $\llb \delta \rrb$ a \emph{striped color}.

We want to color a computable copy of $\omega$ so that in the cohesive power by a given cohesive set, between any two distinct non-standard elements there are elements of every solid color plus at least one element of a striped color.  In this situation, we say that the cohesive power is \emph{colorful}.

\begin{Definition}[\cite{CohPowJournal}*{Definition~5.2}]\label{def-PowColorful}
Let $\mc{O} = (L, \Nb; \prec_\mc{L}, F)$ be a computable colored copy of $\omega$, and let $\mc{L}$ denote $(L; \prec_\mc{L})$.  Let $C$ be a cohesive set.  The cohesive power $\prod_C \mc{O}$ is \emph{colorful} if the following items hold.
\begin{enumerate}[(1)]
\item\label{it-ColorfulStd} For every pair of non-standard elements $[\varphi], [\psi] \in |\prod_C \mc{L}|$ with $[\psi] \prec_{\prod_C \mc{L}} [\varphi]$ and every solid color $\llb \delta \rrb \in |\prod_C \Nb|$, there is a $[\theta] \in |\prod_C \mc{L}|$ with $[\psi] \prec_{\prod_C \mc{L}} [\theta] \prec_{\prod_C \mc{L}} [\varphi]$ and $F^{\prod_C \mc{O}}([\theta]) = \llb \delta \rrb$.

\medskip

\item\label{it-ColorfulNonstd} For every pair of non-standard elements $[\varphi], [\psi] \in |\prod_C \mc{L}|$ with $[\psi] \prec_{\prod_C \mc{L}} [\varphi]$, there is a $[\theta] \in |\prod_C \mc{L}|$ with $[\psi] \prec_{\prod_C \mc{L}} [\theta] \prec_{\prod_C \mc{L}} [\varphi]$ where $F^{\prod_C \mc{O}}([\theta])$ is a striped color.
\end{enumerate}
\end{Definition}

In Definition~\ref{def-PowColorful}, item~\ref{it-ColorfulStd} implies item~\ref{it-ColorfulNonstd} when the cohesive set is $\Delta_2$.  Let $\mc{O} = (L, \Nb; \prec_\mc{L}, F)$ be a computable colored copy of $\omega$, let $C$ be a $\Delta_2$ cohesive set, and suppose that $\prod_C \mc{O}$ satisfies Definition~\ref{def-PowColorful} item~\ref{it-ColorfulStd}.  Let $[\varphi]$ and $[\psi]$ be non-standard elements of $\prod_C \mc{L}$ with $[\psi] \prec_{\prod_C \mc{L}} [\varphi]$.  Let $\delta_i \colon \Nb \to \Nb$ be the constant function with value $i$ for each $i$, so that $(\llb \delta_i \rrb : i \in \Nb)$ is a uniformly computable sequence of all the solid colors of $\prod_C \Nb$.  Now consider the computable $\Sigma_1$-type $p(x; [\varphi], [\psi], \llb \delta_0 \rrb, \llb \delta_1 \rrb, \dots)$ consisting of the formula expressing $[\psi] \prec x \prec [\varphi]$ and, for each $i$, the formula expressing $F(x) \neq \llb \delta_i \rrb$.  The set of formulas $p$ is indeed a type because every solid color occurs between $[\varphi]$ and $[\psi]$.  Thus $p$ is realized by some $[\theta] \in |\prod_C \mc{L}|$ by Theorem~\ref{thm-SatDelta2} item~\ref{it-nDecRecSatStrDelta2}.  Then $[\psi] \prec_{\prod_C \mc{L}} [\theta] \prec_{\prod_C \mc{L}} [\varphi]$, and $F^{\prod_C \mc{O}}([\theta])$ is a striped color because it is not equal to any solid color.  Thus $\prod_C \mc{O}$ also satisfies Definition~\ref{def-PowColorful} item~\ref{it-ColorfulNonstd}.

Given a $\Pi_1$ cohesive set $C$, we can compute a colored copy $\mc{O}$ of $\omega$ such that $\prod_C \mc{O}$ is colorful by~\cite{CohPowJournal}*{Theorem~5.3}.  In Section~\ref{sec-CohPowDelta2}, we compute a single colored copy $\mc{O}$ of $\omega$ such that $\prod_C \mc{O}$ is colorful for \emph{every} $\Delta_2$ cohesive set $C$.

If $\mc{O} = (R, \Nb; \prec_\mc{R}, F)$ is a computable colored copy of $\omega$, $\mc{R}$ denotes $(R; \prec_\mc{R})$, and $C$ is a cohesive set for which $\prod_C \mc{O}$ is colorful, then $\prod_C \mc{R} \iso \omega + \eta$.  This is because the standard elements of $\prod_C \mc{R}$ form an initial segment of order-type $\omega$ by Lemma~\ref{lem-StdInitSeg}, the non-standard elements have no endpoints by Lemma~\ref{lem-NonStdNoEnd}, and the colorfulness of $\prod_C \mc{O}$ implies that the non-standard elements of $\prod_C \mc{R}$ are dense.  In~\cite{CohPowJournal}*{Section~6}, order-types other than $\omega + \eta$ are achieved by starting with $\mc{R}$ and computing another copy $\mc{L}$ of $\omega$ by replacing each element of $\mc{R}$ by some finite linear order depending on the element's color and the desired order-type.  The technique for producing $\mc{L}$ from $\mc{O}$ does not depend on the cohesive set $C$:  so long as $\prod_C \mc{O}$ is colorful, the cohesive power $\prod_C \mc{L}$ has the desired order-type.

\begin{Lemma}[\cite{CohPowJournal}*{Lemmas~6.5 and~6.9}]\label{lem-Shuffle}
Let $X \subseteq \Nb \setminus \{0\}$ be a Boolean combination of $\Sigma_2$ sets, thought of as a set of finite order-types.  Let $\mc{O}$ be a computable colored copy of $\omega$.
\begin{itemize}
\item There is a computable copy $\mc{L}$ of $\omega$ (constructed from $\mc{O}$) such that for every cohesive set $C$, if $\prod_C \mc{O}$ is colorful, then $\prod_C \mc{L}$ has order-type $\omega + \bm{\sigma}(X \cup \{\omega + \zeta\eta + \omega^*\})$.

\medskip

\item Moreover, if $X$ is finite and non-empty, then there is also a computable copy $\mc{L}$ of $\omega$ (constructed from $\mc{O}$) such that for every cohesive set $C$, if $\prod_C \mc{O}$ is colorful, then $\prod_C \mc{L}$ has order-type $\omega + \bm{\sigma}(X)$.
\end{itemize}
\end{Lemma}

We emphasize that the linear order $\mc{L}$ in the conclusion of Lemma~\ref{lem-Shuffle} is not the underlying linear order of $\mc{O}$, but it is obtained from the underlying linear order of $\mc{O}$ by replacing its elements by finite linear orders.  Theorem~\ref{thm-BCSigma2ShuffleOld} is achieved by starting with a $\Pi_1$ cohesive set $C$, computing a colored copy $\mc{O}$ of $\omega$ such that $\prod_C \mc{O}$ is colorful as in~\cite{CohPowJournal}*{Theorem~5.3}, and then applying Lemma~\ref{lem-Shuffle}.

\section{Infinite \texorpdfstring{$\Pi_1$}{Pi\_1} sets without \texorpdfstring{$\Delta_2$}{Delta\_2} cohesive subsets}\label{sec-NoCOH}

The lattice of c.e.\ sets modulo finite difference ordered by $\subseteq^*$ is a major object of study in computability theory, and we refer the reader to~\cites{RogersBook, SoareBookRE} for a thorough treatment of the subject.  Recall that a c.e.\ set $M$ is called \emph{maximal} if $M \subsetneq^* \Nb$ and there is no c.e.\ set $X$ with $M \subsetneq^* X \subsetneq^* \Nb$.  Friedberg~\cite{FriedbergThreeTheorems} shows that maximal sets exist, and D.\ Martin~\cite{Martin} shows that there is a c.e.\ set $A \subsetneq^* \Nb$ that has no maximal superset.  It follows from the definitions that for any set $M \subseteq \Nb$, $M$ is maximal if and only if $\ol{M}$ is an infinite $\Pi_1$ cohesive set.  Via this correspondence, Martin's result may be rephrased as stating that there is an infinite $\Pi_1$ set that has no $\Pi_1$ cohesive subset.  There is also an infinite $\Pi_1$ set that has no $\Delta_2$ r-cohesive subset, which follows from work of Lerman, Shore, and Soare~\cite{LermanShoreSoare}.

Let $X$ and $Y$ be c.e. sets where $Y \subseteq X$ and $X \setminus Y$ is infinite.  In this situation,
\begin{itemize}
\item $Y$ is \emph{r-maximal} in $X$ if $X \setminus Y$ is r-cohesive;
\smallskip
\item $Y$ is a \emph{major} subset of $X$ if $\forall \text{ c.e.\ sets } W\; (\ol{X} \subseteq^* W \imp \ol{Y} \subseteq^* W)$;
\smallskip
\item $Y$ is an \emph{r-maximal major} subset of $X$ if $Y$ is both r-maximal in $X$ and a major subset of $X$.
\end{itemize}

Lerman, Shore, and Soare~\cite{LermanShoreSoare}*{Theorem~1.2} show that a c.e.\ set $X$ has an r-maximal major subset if and only if it has a $\Delta_3$ \emph{preference function}, where preference functions are defined as follows.

\begin{Definition}[\cite{LermanShoreSoare}*{Definition~1.1}]
For each $e$, let $R^{e,1} = \{n : (\forall m \leq n\; \varphi_e(m)\da) \;\land\; \varphi_e(n) = 1\}$, and let $R^{e,0} = \ol{R^{e,1}}$.  Notice that $(R^{e,1})_{e \in \Nb}$ is a uniformly c.e.\ sequence of the computable sets.  An \emph{inclination function}\footnote{We introduce the term \emph{inclination function} here for expository purposes.} for a set $X$ is a function $h \colon \Nb \to \{0,1\}$ such that for all $n$, the set $X \cap \bigcap_{e < n}R^{e,h(e)}$ is infinite.  A \emph{preference function} for a set $X$ is a function $h \colon \Nb \to \{0,1\}$ that is an inclination function for both $X$ and $\ol{X}$.
\end{Definition}

Lerman, Shore, and Soare observe that if $X$ is a simple set (i.e., $X$ is c.e.\ and $\ol{X}$ is infinite but has no infinite c.e.\ subset), then every inclination function for $\ol{X}$ is automatically a preference function for $X$.  Thus a simple set $X$ has a preference function if and only if $\ol{X}$ has an inclination function.

It is not hard to show that an infinite $\Delta_2$ set has a $\Delta_2$ r-cohesive subset if and only if it has a $\Delta_3$ inclination function.

\begin{Proposition}\label{prop-rCohInc}
Let $B \subseteq \Nb$ be an infinite $\Delta_2$ set.  Then $B$ has a $\Delta_2$ r-cohesive subset if and only if $B$ has a $\Delta_3$ inclination function.
\end{Proposition}

\begin{proof}
First suppose that $h \leqT 0''$ is an inclination function for $B$.  By the limit lemma, there is a function $g \colon \Nb^2 \to \{0,1\}$ with $g \leqT 0'$ such that $\lim_s g(e, s) = h(e)$ for every $e$.  We compute an increasing enumeration $c_0 < c_1 < c_2 < \cdots$ of an r-cohesive set $C \subseteq B$ from $0'$ using that $B \leqT 0'$, that $g \leqT 0'$, and the fact the sequence $(R^{e,i} : e \in \Nb, i \in \{0,1\})$ is uniformly computable from $0'$.  The characteristic function of $C$ may then be computed from its increasing enumeration, so $C$ is the desired $\Delta_2$ r-cohesive subset of $B$.

Let $c_0$ be the least element of $B$.  Suppose we have already determined $c_0 < c_1 < \cdots < c_n$.  To find $c_{n+1}$, search for the least pair $\la s, c \ra$ with $s > n$ and $c > c_n$ and $c \in B \cap \bigcap_{e < n}R^{e, g(e,s)}$.  Then let $c_{n+1} = c$.  Such a pair $\la s, c \ra$ always exists because if $s$ is sufficiently large, then $\forall e < n \; (g(e,s) = h(e))$, in which case $B \cap \bigcap_{e < n}R^{e, g(e,s)}$ is infinite.

To check that $C$ is r-cohesive, let $X$ be a computable set, and let $e$ be such that $X = R^{e,1}$.  Let $s_0$ be large enough so that $g(e,s) = h(e)$ for all $s \geq s_0$.  Then $c_{n+1}$ is chosen from $R^{e,h(e)}$ whenever $n \geq s_0$.  Thus if $h(e) = 1$, then $C \subseteq^* R^{e,1} = X$; and if $h(e) = 0$, then $C \subseteq^* R^{e,0} = \ol{X}$.  So $C$ is r-cohesive.

Conversely, suppose that $B$ has a $\Delta_2$ r-cohesive subset $C$.  The sets $\{e : C \subseteq^* R^{e,0}\}$ and $\{e : C \subseteq^* R^{e,1}\}$ are each $\Sigma_3$, and they are complements because $C$ is r-cohesive.  Thus the sets are $\Delta_3$.  Let $h$ be the characteristic function of $\{e : C \subseteq^* R^{e,1}\}$.  Then $C \subseteq^* \bigcap_{e < n}R^{e,h(e)}$ for every $n$, so $B \cap \bigcap_{e < n}R^{e,h(e)}$ is infinite for every $n$ because $C \subseteq B$.  Thus $h$ is a $\Delta_3$ inclination function for $B$.
\end{proof}

Lerman, Shore, and Soare~\cite{LermanShoreSoare}*{Theorem~2.7} show that there is a simple (indeed, hyperhypersimple) set $X$ that has no $\Delta_3$ preference function.  Therefore $\ol{X}$ is an infinite $\Pi_1$ set that has no $\Delta_3$ inclination function and hence no $\Delta_2$ r-cohesive subset.

We give a direct construction of an infinite $\Pi_1$ set with no $\Delta_2$ p-cohesive subset.  To organize this, we make use of a \emph{uniform sequence containing all $\Delta_2$-approximations} and a computable \emph{sequence of staggered partitions}.  Recall that a computable function $g \colon \Nb^2 \to \{0,1\}$ is a $\Delta_2$-approximation to a set $D \subseteq \Nb$ if $\forall n \, (\lim_s g(n,s) = D(n))$ and that, by the limit lemma, $D$ is $\Delta_2$ if and only if it has a $\Delta_2$-approximation.

\begin{Definition}\label{def-UniformDelta2Approx}
A \emph{uniform sequence containing all $\Delta_2$-approximations} is a total computable function $g \colon \Nb^3 \to \{0,1\}$ such that for every $\Delta_2$ set $D$, there is an $e$ for which the function $g_e \colon \Nb^2 \to \{0,1\}$ given by $g_e(n,s) = g(e,n,s)$ is a $\Delta_2$-approximation to $D$.
\end{Definition}

A uniform sequence $g \colon \Nb^3 \to \{0,1\}$ containing all $\Delta_2$-approximations may be computed as follows.  On input $(e,n,s)$, search for the greatest $t \leq s$ such that $\varphi_{e,s}(n,t)\da$.  If $\varphi_{e,s}(n,t) = 1$, then output $g(e,n,s) = 1$.  If $\varphi_{e,s}(n,t) \neq 1$ (or if there is no $t \leq s$ such that $\varphi_{e,s}(n,t)\da$), then output $g(e,n,s) = 0$.  Notice that $g$ is total.  Write $g_e(n,s)$ for $g(e,n,s)$ for every $e$, $n$, and $s$.  If $\varphi_e$ is total, $\{0,1\}$-valued, and $\lim_s \varphi_e(n,s)$ exists for every $n$, then $\lim_s g_e(n,s)$ exists and equals $\lim_s \varphi_e(n,s)$ for every $n$.  Thus if $D$ is a $\Delta_2$ set, then some $\varphi_e(n,s)$ is a $\Delta_2$-approximation to $D$, in which case $g_e$ is also a $\Delta_2$-approximation to $D$.

If $g$ is a uniform sequence containing all $\Delta_2$-approximations, then every $g_e$ is total, and every $\Delta_2$ set is approximated by some $g_e$.  However, it is not the case that every $g_e$ is a $\Delta_2$-approximation.  There are many $e$ and $n$ for which $\lim_s g_e(n,s)$ does not exist.

For each $n \in \Nb$, let $\{0,1\}^n$ denote the set of binary sequences of length $n$.

\begin{Definition}\label{def-StagPart}
A \emph{sequence of staggered partitions} is a sequence $(A^{i,0}, A^{i,1})_{i \in \Nb}$ of pairs of subsets of $\Nb$ such that
\begin{itemize}
\item for each $i$, $A^{i,0}$ and $A^{i,1}$ partition $\Nb$ into two parts (i.e., $A^{i,1} = \ol{A^{i,0}}$), and

\medskip

\item $\forall n \; \forall \sigma \in \{0,1\}^n \; \bigl(\text{$\bigcap_{i < n} A^{i, \sigma(i)}$ is infinite}\bigr)$.
\end{itemize}

An infinite $C \subseteq \Nb$ is \emph{cohesive} for a sequence of staggered partitions $\vec{A} = (A^{i,0}, A^{i,1})_{i \in \Nb}$, or \emph{$\vec{A}$-cohesive}, if $\forall i\, (C \subseteq^* A^{i,0} \;\lor\; C \subseteq^* A^{i,1})$.
\end{Definition}
Of course, $C$ being cohesive for the sequence $(A^{i,0}, A^{i,1})_{i \in \Nb}$ of staggered partitions means the same thing as $C$ being cohesive for the sequence $(A^{i,0})_{i \in \Nb}$ (or for the sequence $(A^{i,1})_{i \in \Nb}$).

We may compute a sequence $(A^{i,0}, A^{i,1})_{i \in \Nb}$ of staggered partitions as follows.  Given $i$, partition $\Nb$ into successive pieces of size $2^i$, let $A^{i,0}$ consist of every other piece, and let $A^{i,1} = \ol{A^{i,0}}$.  Indeed, for $i \in \Nb$ and $j \in \{0,1\}$, we may take $A^{i,j}$ to be the set of numbers whose binary expansions have bit $j$ in position $i$.  Then each $A^{i,j}$ is primitive recursive.  In fact, the sequence $(A^{i,0}, A^{i,1})_{i \in \Nb}$ is uniformly primitive recursive.  We show that for any uniformly computable sequence $\vec{A}$ of staggered partitions, there is an infinite $\Pi_1$ set with no $\Delta_2$ $\vec{A}$-cohesive subset.  If all the sets of $\vec{A}$ are from a particular class, then it follows that there is an infinite $\Pi_1$ set with no $\Delta_2$ subset that is cohesive for that class.  There is a uniformly primitive recursive sequence of staggered partitions, so there is an infinite $\Pi_1$ set with no $\Delta_2$ p-cohesive subset.

\begin{Theorem}\label{thm-NoCOH-helper}
Let $\vec{A} = (A^{i,0}, A^{i,1})_{i \in \Nb}$ be a uniformly computable sequence of staggered partitions.  Then there is an infinite $\Pi_1$ set with no $\Delta_2$ $\vec{A}$-cohesive subset.
\end{Theorem}

\begin{proof}
We implement a movable markers style construction of an infinite $\Pi_1$ set $B$ with no $\Delta_2$ $\vec{A}$-cohesive subset.  For each $i, s \in \Nb$, let $m_{i,s}$ denote the position of marker $m_i$ at stage $s$.  Say that a number $n$ is \emph{marked} at stage $s$ if $n = m_{i,s}$ for some $i$, and say that $n$ is \emph{unmarked} at stage $s$ otherwise.  At stage $0$, we start with $m_{i,0} = i$ for each $i$.  At stage $s > 0$, for each $i$, we move marker $m_i$ to the previous position of marker $m_j$ for some $j \geq i$.  That is, for each $i$, we set $m_{i,s} = m_{j,s-1}$ for some $j \geq i$.  When we move the markers, we do so in such a way as to maintain that at every stage $s$, only finitely many numbers are unmarked and $\forall i \, (m_{i, s} < m_{i+1, s})$.  If a number is unmarked at stage $s$, then it is unmarked at all later stages.  The set of numbers that are ever unmarked during the course of the construction is $\Sigma_1$.  Its complement is the desired $\Pi_1$ set $B$, which consists of the final positions of all the markers.

Let $g \colon \Nb^3 \to \{0,1\}$ be a uniform sequence containing all $\Delta_2$-approximations.  As above, write $g_e(n,s)$ for $g(e,n,s)$.

If $C$ is an $\vec{A}$-cohesive set, then for every $e$, either $C \subseteq^* A^{e,0}$ or $C \subseteq^* A^{e,1}$.  The goal of the construction is to arrange, for every $\Delta_2$ $\vec{A}$-cohesive set $C$, that if $C \subseteq^* A^{e,a}$ for $a \in \{0,1\}$, then almost every marker eventually settles on a member of $A^{e, 1-a}$.  If this is achieved, then $C \subseteq^* A^{e,a}$ and $B \subseteq^* A^{e, 1-a}$, so $B \cap C$ is finite.  Thus $B$ has no $\Delta_2$ $\vec{A}$-cohesive subset.

For the purposes of this construction, think of a pair $\la e, N \ra$ as coding a guess that $g_e$ is a $\Delta_2$-approximation of an $\vec{A}$-cohesive set $C$ and that $N$ is a threshold by which either $\forall n \geq N \; (n \in C \imp n \in A^{e,0})$ or $\forall n \geq N \; (n \in C \imp n \in A^{e,1})$.  

Say that pair $\la e, N \ra$ is \emph{active} at stage $s$ if there is a side $a \in \{0,1\}$ of the partition $(A^{e,0}, A^{e,1})$ along with a witness $w$ with $N < w < s$ meeting the following conditions.
\begin{enumerate}[(a)]
\item $g_e(w,s) = 1$.

\medskip

\item $\forall n \; \bigl((N \leq n \leq w \;\land\; g_e(n,s) = 1) \;\imp\; n \in A^{e,a}\bigr)$.
\end{enumerate}
Notice that if $\la e, N \ra$ is active at stage $s$, then there is a unique $a \in \{0,1\}$ meeting these conditions.  Call this $a$ the \emph{active side}, and call the largest witness $w$ (with $N < w < s$) the \emph{activity witness} for $\la e, N \ra$ at stage $s$.

Suppose that pair $q = \la e, N \ra$ is active at stage $s$ with active side $a$ and activity witness $w$.  Then it looks like $g_e$ approximates a $\Delta_2$ set $C$ with $C \subseteq^* A^{e,a}$, so pair $q$ wants to move the markers into $A^{e,1-a}$.  We need to ensure that each marker moves only finitely often, so we stipulate that $q$ may only move marker $m_i$ if $i > q$.  However, $g_e$ may not approximate a $\Delta_2$ set at all, let alone a $\Delta_2$ cohesive set.  If $g_e$ is particularly ill-behaved, then $q$ may change its active side infinitely often and therefore may want to move the markers $m_i$ with $i > q$ back and forth between $A^{e,0}$ and $A^{e,1}$ infinitely often.  To combat this problem, we associate a counter $\ct(q)$ to pair $q$ that counts the number of times that the active side of $q$ changes.  Then we also stipulate that $q$ may only move marker $m_i$ if $i > \ct(q)$.  Finally, the larger the activity witness $w$ gets, the more evidence we have that $g_e$ indeed approximates a $\Delta_2$ set $C$ with $C \subseteq^* A^{e,a}$, and we require seeing this evidence before moving marker $m_i$.  That is, we also stipulate that $q$ may only move marker $m_i$ if $m_i < w$.  The construction now proceeds as follows.

At stage $0$, initialize $m_{i,0} = i$ for each $i$, and initialize $\ct(q) = 0$ for every pair $q$.

At stage $s > 0$, consider all pairs $q < s$.  First update the counters as follows.  For each pair $q < s$, if $q$ is active at stage $s$ and either this is the first stage at which $q$ is active or the active side of $q$ is different than it was at the previous stage at which $q$ was active, then update counter $\ct(q)$ to $\ct(q) + 1$.

Now let $i < s$ be the least number for which there is a pair $q = \la e, N \ra < s$ meeting the following conditions.
\begin{enumerate}[(1)]
\item\label{it:active} $q$ is active at stage $s$ with active side $a$ and activity witness $w$ for some $a$ and $w$.

\medskip

\item\label{it:pri} $q < i$ and $\ct(q) < i$.

\medskip

\item\label{it:wit} $m_{i, s-1} < w$.

\medskip

\item\label{it:side} $m_{i, s-1} \in A^{e, a}$.
\end{enumerate}
If there is no such $i$, then put $m_{k, s} = m_{k, s-1}$ for each $k$ and go on to the next stage.  If there is such an $i$, then let $q$ be the least witnessing pair.  Let 
\begin{align*}
R &= \{\pi_0(p) : p \leq q \;\land\; \ct(p) < i\}.
\end{align*}
For each $y \in R$, let $v$ be the greatest activity witness yet achieved by any pair of the form $\la y, M \ra \leq q$ with $\ct(\la y, M \ra) < i$, let $t$ be the most recent stage at which $v$ was achieved, let $M$ be least such that $\la y, M \ra \leq q$ achieved activity witness $v$ at stage $t$, and let $a_y$ be the corresponding active side.  If no pair $\la y, M \ra \leq q$ has yet been active, then let $a_y = 0$.  Now let $j > i$ be the least number such that
\begin{align*}
m_{j, s-1} \in \bigcap_{y \in R} A^{y, 1-a_y}.
\end{align*}
Such a $j$ exists because the intersection is infinite and almost every number is marked at stage $s-1$.  Advance the markers by setting $m_{k,s} = m_{k,s-1}$ for all $k < i$ and by setting $m_{k,s} = m_{k+j-i, s-1}$ for all $k \geq i$.  In this way we maintain that only finitely many numbers are unmarked at stage $s$ and that $\forall k\, (m_{k,s} < m_{k+1,s})$.  We say that pair $q$ has now \emph{moved marker $m_i$}.  This completes the construction.

We show that for each $i$, marker $m_i$ moves only finitely often.  Consider marker $m_i$, and inductively assume that there is a stage $s_0$ after which no marker $m_k$ with $k < i$ ever moves.  Thus after stage $s_0$, marker $m_i$ is never moved on account of the movement of a marker $m_k$ with $k < i$.  So if marker $m_i$ moves after stage $s_0$, it is because there is a pair $q$ that meets conditions~\ref{it:active}--\ref{it:side} for $i$ and hence moves $m_i$.  If $q \geq i$, then $q$ always fails condition~\ref{it:pri} and hence never moves marker $m_i$.  Thus it suffices to show that each pair $q < i$ moves marker $m_i$ only finitely often.  Fix $q_0 < i$, and inductively assume that there is a stage $s_1 > s_0$ after which no pair $q < q_0$ ever moves marker $m_i$.  There are now several cases.

If $q_0$ is active only finitely often, then condition~\ref{it:active} eventually always fails and therefore $q_0$ moves $m_i$ only finitely often.

If $q_0$ changes its active side at least $i$ many times, then condition~\ref{it:pri} eventually always fails on account of $\ct(q_0) \geq i$.  Thus $q_0$ moves $m_i$ only finitely often.

If $q_0$ does not achieve arbitrarily large activity witnesses, then there is a bound $W$ such that the activity witness of $q_0$ is below $W$ whenever $q_0$ is active.  When marker $m_i$ moves, it moves to mark a larger number than it did previously.  That is, if $m_i$ moves at stage $s$, then $m_{i, s} > m_{i, s-1}$.  Marker $m_i$ can move only finitely often before achieving $m_{i, s} \geq W$.  If $m_i$ does achieve $m_{i, s} \geq W$ at some stage $s$, then condition~\ref{it:wit} fails at all later stages.  Thus $q_0$ moves $m_i$ only finitely often.

Finally, suppose that pair $q_0 = \la e_0, N_0 \ra$ is active infinitely often, that $\ct(q_0) < i$ at all stages, and that $q_0$ achieves arbitrarily large activity witnesses.  As $\ct(q_0) < i$ at all stages, $q_0$ changes its active side only finitely often and therefore eventually settles on some final active side $a \in \{0,1\}$.  That is, there is a stage $s_2 > s_1$ such that $q_0$ has active side $a$ whenever it is active at a stage $s > s_2$.

Consider the pairs of the form $\la e_0, M \ra < i$.  Let $s_3 > s_2$ be large enough so that 
\begin{itemize}
\item each pair $\la e_0, M \ra < i$ that eventually achieves $\ct(\la e_0, M \ra) \geq i$ has done so by stage $s_3$, and

\medskip

\item each pair $\la e_0, M \ra < i$ with $\ct(\la e_0, M \ra) < i$ at all stages has had $\ct(\la e_0, M \ra)$ settle on its final value by stage $s_3$.
\end{itemize}
The following Claims~\ref{claim-AgreePhelper}--\ref{claim-MoveOut} help us complete the argument that pair $q_0 = \la e_0, N_0 \ra$ moves marker $m_i$ only finitely often.

\begin{ClaimP}\label{claim-AgreePhelper}
Consider pairs $\la e, M_0 \ra$ and $\la e, M_1 \ra$ with $M_0 \leq M_1$, where $\la e, M_0 \ra$ is active at stage $s$ with active side $b \in \{0,1\}$ and activity witness $w > M_1$.  Then $\la e, M_1\ra$ is also active at stage $s$ with active side $b$.
\end{ClaimP}

\begin{proof}[Proof of Claim]
Pair $\la e, M_0 \ra$ is active at stage $s$ with active side $b$ and activity witness $w > M_1$.  Therefore $g_e(w,s) = 1$; and $n \in A^{e,b}$ whenever $M_0 \leq n \leq w$ and $g_e(n,s) = 1$.  Thus $b$ and $w$ also witness that that $\la e, M_1 \ra$ is active at stage $s$ with active side $b$ because $M_0 \leq M_1 < w$.
\end{proof}

\begin{ClaimP}\label{claim-AgreeP}
If a pair $\la e_0, M \ra < i$ is active at stage $s > s_3$ with $\ct(\la e_0, M \ra) < i$ and an activity witness $w > i$, then its active side is $a$, the final active side of $q_0$.
\end{ClaimP}

\begin{proof}[Proof of Claim]
Suppose that $\la e_0, M \ra < i$ is active at stage $s > s_3$ with $\ct(\la e_0, M \ra) < i$, active side $b$, and activity witness $w > i$.

First suppose that $M \leq N_0$, and note that $N_0 \leq \la e_0, N_0 \ra < i < w$.  Then by Claim~\ref{claim-AgreePhelper}, pair $q_0 = \la e_0, N_0 \ra$ is also active at stage $s$ with active side $b$.  However, $q_0$ has active side $a$ at stage $s$ because $s > s_3 > s_2$.  Thus $b = a$, so $\la e_0, M \ra$ has active side $a$ at stage $s$.

Suppose instead that $M > N_0$.  Pair $q_0$ achieves arbitrarily large activity witnesses, so there is a stage $t > s$ at which $q_0$ is active with active side $a$ (as $t > s_2$) and activity witness $v > i$.  Note that $M \leq \la e_0, M \ra < i < v$.  By Claim~\ref{claim-AgreePhelper}, pair $\la e_0, M \ra$ is also active at stage $t$ with active side $a$.  However, if $b \neq a$, this means that $\la e_0, M \ra$ changes its active side from $b$ to $a$, thereby incrementing $\ct(\la e_0, M \ra)$, at some stage between $s$ and $t$.  This contradicts that $\ct(\la e_0, M \ra)$ had already stabilized by stage $s_3 < s$.  Therefore $b = a$, and pair $\la e_0, M \ra$ must have had active side $a$ at stage $s$.
\end{proof}

Let $s_4 > s_3$ be the first stage at which pair $q_0$ achieves an activity witness $w > i$ that is also greater than the maximum activity witness achieved by the pairs of the form $\la e_0, M \ra$ with $\la e_0, M \ra < i$ at stages $s \leq s_3$.  Such an $s_4$ exists by the assumption that $q_0$ achieves arbitrarily large activity witnesses.

\begin{ClaimP}\label{claim-MoveOut}
If marker $m_i$ moves at a stage $s > s_4$, then $m_i$ moves to mark an element of $A^{e_0, 1-a}$, where $a$ is the final active side of $q_0$.  That is, if $s > s_4$ and $m_{i,s} > m_{i, s-1}$, then $m_{i,s} \in A^{e_0, 1-a}$.
\end{ClaimP}

\begin{proof}[Proof of Claim]
Suppose that $m_i$ moves at stage $s > s_4$.  As $s_4 > s_1$, it must be that $m_i$ is moved by a pair $q$ with $q_0 \leq q < i$.  As $q_0 = \la e_0, N_0 \ra$ and $\ct(q_0) < i$, index $e_0$ is in the set $R$ used by $q$ to move $m_i$.  The action of $q$ to move $m_i$ at stage $s$ thus involves choosing an active side $a_{e_0}$ for index $e_0$.  We show that $a_{e_0} = a$ at stage $s$.  It then follows that $m_{i,s} \in A^{e_0, 1-a}$ because, at stage $s$, $m_{i,s}$ is set to $m_{j, s-1}$ for an $m_{j, s-1} \in A^{e_0, 1-a}$.

At stage $s$, the side $a_{e_0}$ is determined by considering the pairs $\la e_0, M \ra \leq q$ with $\ct(\la e_0, M \ra) < i$, finding the greatest activity witness $v$ yet achieved by any such pair, and by finding the most recent stage $t$ at which this activity witness was achieved.  Pair $q_0 = \la e_0, N_0 \ra$ is among the considered pairs, and, at stage $s_4$, $q_0$ achieves an activity witness $w > i$ that is greater than any activity witness ever achieved by a pair $\la e_0, M \ra \leq q < i$ at a stage $r \leq s_3$.  Thus it must be that $v > i$, and $v$ must be achieved by some $\la e_0, M \ra \leq q < i$ with $\ct(\la e_0, M \ra) < i$ at some stage $t > s_3$.  By Claim~\ref{claim-AgreeP}, $a$ is the active side for any such pair $\la e_0, M \ra$ that is active at a stage $t > s_3$ with activity witness $v > i$.  Therefore $a_{e_0}$ is chosen to be $a$ at stage $s$.
\end{proof}

By Claim~\ref{claim-MoveOut}, whenever marker $m_i$ moves after stage $s_4$, it moves to mark an element of $A^{e_0, 1-a}$.  Thus if $m_i$ moves at stage $s > s_4$, then $\forall t \geq s \; (m_{i,t} \in A^{e_0, 1-a})$.  In this case condition~\ref{it:side} fails for $q_0$ at all stages $t > s$, and therefore $q_0$ cannot move $m_i$ at any stage $t > s$.  Thus pair $q_0$ moves marker $m_i$ at most once after stage $s_4$, and therefore $q_0$ moves $m_i$ only finitely often.  This completes the proof that pair $q_0$ moves marker $m_i$ only finitely often, which completes the proof that marker $m_i$ moves only finitely often.  Thus the construction indeed produces an infinite $\Pi_1$ set $B$.

We finish the proof by showing that if $C$ is a $\Delta_2$ $\vec{A}$-cohesive set, then the final position of almost every marker is in $\ol{C}$.  This shows that $B \cap C$ is finite and therefore that $C$ is not an $\vec{A}$-cohesive subset of $B$.  It follows that $B$ has no $\Delta_2$ $\vec{A}$-cohesive subset.

Let $C$ be a $\Delta_2$ $\vec{A}$-cohesive set, and let $e$ be such that $g_e$ is a $\Delta_2$-approximation to $C$.  By $\vec{A}$-cohesiveness, either $C \subseteq^* A^{e,0}$ or $C \subseteq^* A^{e,1}$.  Let $a \in \{0,1\}$ be such that $C \subseteq^* A^{e,a}$, and let $N$ be least such that $\forall n \geq N \; (n \in C \imp n \in A^{e,a})$.  Let $q = \la e, N \ra$.

\begin{ClaimP}\label{claim-GoodPair}
Pair $q$ achieves arbitrarily large activity witnesses and eventually settles on active side $a$.
\end{ClaimP}

\begin{proof}[Proof of Claim]
We first show that $q$ achieves arbitrary large activity witnesses.  Given any number $W > N$, let $w$ be the least number with $w > W$ and $w \in C$.  Let $s > w$ be large enough so that $g_e(n,s) = C(n)$ for all $n \leq w$.  Then $q$ is active at stage $s$ with active side $a$ and activity witness $w$ or greater.  Thus $q$ achieves arbitrarily large activity witnesses.

Now we show that $q$ eventually settles on active side $a$.  Let $n$ be the least number with $n > N$ and $n \in C$.  Let $s_0$ be large enough so that $\forall m \leq n \; \forall s \geq s_0 \; (g_e(m,s) = C(m))$.  Then if $q$ is active at a stage $s > s_0$, it must use an activity witness $w \geq n$, in which case its active side must be $a$ because $g_e(n,s) = 1$ and $n \in A^{e,a}$.  That is, $q$ has active side $a$ whenever it is active at a stage later than $s_0$.  Thus $q$ eventually settles on active side $a$.
\end{proof}

By Claim~\ref{claim-GoodPair}, pair $q$ changes its active side only finitely often, so there is a stage $s_0$ by which $\ct(q)$ has reached its final value, which we also denote $\ct(q)$.  Consider an $i$ with $q < i$ and $\ct(q) < i$.  Let $s_1 > s_0$ be a stage by which marker $m_i$ has stopped moving:  $\forall s > s_1 \; (m_{i, s} = m_{i, s_1})$.  Then $m_{i, s_1} \in A^{e, 1-a}$.  If instead $m_{i, s_1} \in A^{e, a}$, then, by Claim~\ref{claim-GoodPair}, let $s > s_1$ be a stage at which $q$ is active with active side $a$ and activity witness $w > m_{i, s-1}$.  Then $q$ meets conditions~\ref{it:active}--\ref{it:side} at stage $s$ because $m_{i, s-1} = m_{i, s_1} \in A^{e,a}$.  Thus marker $m_i$ moves at stage $s$, either directly by some pair or on account of the movement of a marker $m_k$ with $k < i$.  This is a contradiction.  We have shown that marker $m_i$ settles on a member of $A^{e, 1-a}$ whenever $i$ satisfies $q < i$ and $\ct(q) < i$.  We therefore have that $B \subseteq^* A^{e, 1-a}$ and that $C \subseteq^* A^{e,a}$.  Thus $B \cap C$ is finite, as desired.
\end{proof}

\begin{Corollary}\label{cor-No-pCOH}
There is an infinite $\Pi_1$ set with no $\Delta_2$ p-cohesive subset.
\end{Corollary}

\begin{proof}
Apply Theorem~\ref{thm-NoCOH-helper} to a uniformly primitive recursive sequence $\vec{A}$ of staggered partitions.  Then there is an infinite $\Pi_1$ set $B$ with no $\Delta_2$ $\vec{A}$-cohesive subset and hence no $\Delta_2$ p-cohesive subset.
\end{proof}

Corollary~\ref{cor-No-pCOH} is optimal in terms of the arithmetical hierarchy because it follows that there is an infinite $\Pi_1$ set with no $\Sigma_2$ p-cohesive subset, whereas every infinite $\Pi_1$ set has a $\Pi_2$ cohesive subset.  First, an infinite $\Pi_1$ set $B$ with no $\Delta_2$ p-cohesive subset also has no $\Sigma_2$ p-cohesive subset.  Every infinite $\Sigma_2$ set has an infinite $\Delta_2$ subset by the fact that every infinite c.e.\ set has an infinite computable subset relativized to $0'$.  Also, infinite subsets of p-cohesive sets are p-cohesive.  Thus if $C$ were a $\Sigma_2$ p-cohesive subset of $B$, then $C$ would have an infinite $\Delta_2$ subset $D$, which would be a contradictory $\Delta_2$ p-cohesive subset of $B$.  Second, every infinite $\Pi_1$ set has a $\Pi_2$ cohesive subset by the fact that every infinite computable set has a $\Pi_1$ cohesive subset (see~\cite{RogersBook}*{Theorem~XI}, for example) relativized to $0'$.  In fact, relativizing to $0'$ yields that every infinite $\Delta_2$ set has a $\Pi_2$ subset that is cohesive for the collection of $\Sigma_2$ sets.

Lastly, we paste together a few facts from the literature to observe that the collection of infinite $\Pi_1$ sets without $\Delta_2$ cohesive subsets does not coincide with the collection of infinite $\Pi_1$ sets without $\Pi_1$ cohesive subsets.  Specifically, we observe that there is an infinite $\Pi_1$ set that does not have a $\Pi_1$ r-cohesive subset but does have a $\Delta_2$ cohesive subset.

Recall that a set $B \subseteq \Nb$ is called \emph{semi-low\textsubscript{2}} if $\{e : \text{$W_e \cap B$ is infinite}\} \leqT 0''$.  We can use the strategy from Jockusch's proof that every set $X$ with $X' \geqT 0''$ computes a cohesive set~\cite{JockuschHHI}*{Theorem~4.1} to show that every infinite $\Delta_2$ set that is semi-low\textsubscript{2} has a $\Delta_2$ cohesive subset.  The proof is also similar to that of Proposition~\ref{prop-rCohInc}.

\begin{Proposition}[Following~\cite{JockuschHHI}*{Theorem~4.1}]\label{prop-Delta2Coh}
Let $B \subseteq \Nb$ be an infinite set that is $\Delta_2$ and semi-low\textsubscript{2}.  Then $B$ has a $\Delta_2$ cohesive subset.
\end{Proposition}

\begin{proof}
Given indices $e_0, \dots, e_n$, we can effectively produce an index $e$ such that $W_e = \bigcap_{i \leq n}W_{e_i}$.  Using this and the fact that $B$ is semi-low\textsubscript{2}, we can define the following function $f \leqT 0''$ by recursion.
\begin{align*}
f(0) &= \text{the least $e$ such that $B \cap W_e$ is infinite}\\
f(n+1) &= \text{the least $e > f(n)$ such that $B \cap \bigcap_{i \leq n}W_{f(i)} \cap W_e$ is infinite.}
\end{align*}
By the limit lemma, there is a function $g \colon \Nb^2 \to \Nb$ with $g \leqT 0'$ such that $\lim_s g(n, s) = f(n)$ for every $n$.  We compute an increasing enumeration $c_0 < c_1 < c_2 < \cdots$ of a cohesive set $C \subseteq B$ from $0'$ using that $B \leqT 0'$, that $g \leqT 0'$, and the fact that the sets $W_e$ are uniformly computable from $0'$.  The characteristic function of $C$ may then be computed from its increasing enumeration, so $C$ is the desired $\Delta_2$ cohesive subset of $B$.

Let $c_0$ be the least element of $B$.  Suppose we have already determined $c_0 < c_1 < \cdots < c_n$.  To find $c_{n+1}$, search for the least pair $\la s, c \ra$ with $s > n$ and $c > c_n$ and $c \in B \cap \bigcap_{i \leq n}W_{g(i,s)}$.  Then let $c_{n+1} = c$.  Such a pair $\la s, c \ra$ always exists because if $s$ is sufficiently large, then $\forall i \leq n \; (g(i,s) = f(i))$, in which case $B \cap \bigcap_{i \leq n}W_{g(i,s)}$ is infinite.

We check that $C$ is cohesive.  Suppose that $e \in \ran(f)$.  Let $i$ be such that $f(i) = e$, and let $s_0$ be large enough so that $g(i,s) = f(i) = e$ for all $s \geq s_0$.  Then $c_{n+1}$ is chosen from $W_{f(i)} = W_e$ whenever $n \geq \max\{i, s_0\}$.  Thus $C \subseteq^* W_e$.  Now suppose that $e \notin \ran(f)$, and let $n$ be least such that $f(n) > e$.  Then $B \cap \bigcap_{i < n}W_{f(i)} \cap W_e$ must be finite because otherwise we would have $f(n) \leq e$.  We just showed that $C \subseteq^* B \cap \bigcap_{i < n}W_{f(i)}$, so $C \cap W_e$ must be finite as well.  Therefore $C$ is cohesive.
\end{proof}

\begin{Proposition}\label{prop-NoPi1butYesDelta2}
There is an infinite $\Pi_1$ set that has no $\Pi_1$ r-cohesive subset but does have a $\Delta_2$ cohesive subset.
\end{Proposition}

\begin{proof}
Let $A$ be Lachlan's hyperhypersimple set with no maximal superset from~\cite{LachlanLattice}.  Recall that a co-infinite c.e.\ set is called r-maximal if its complement is r-cohesive.  Then in fact $A$ has no r-maximal superset because hyperhypersimiplicity is $\subseteq$-upwards closed in the co-infinite c.e.\ sets and because r-maximality and hyperhypersimiplicity together imply maximality by~\cite{SoareBookRE}*{Proposition~4.5}.  So if $A$ had an r-maximal superset $X$, then $X$ would be a maximal superset of $A$, which is a contradiction.  Thus $\ol{A}$ is an infinite $\Pi_1$ set with no $\Pi_1$ r-cohesive subset.  Maass~\cite{Maass} observes that $\ol{A}$ is semi-low\textsubscript{2} (see also~\cite{SoareBookRE}*{Section~XVI.1}), so $\ol{A}$ has a $\Delta_2$ cohesive subset by Proposition~\ref{prop-Delta2Coh}.  Thus $B = \ol{A}$ is an infinite $\Pi_1$ set that has no $\Pi_1$ r-cohesive subset but does have a $\Delta_2$ cohesive subset.
\end{proof}

One may of course wonder if `no $\Pi_1$ r-cohesive subset' can be improved to `no $\Pi_1$ p-cohesive subset' in Proposition~\ref{prop-NoPi1butYesDelta2}.  We did not attempt to determine this.

\section{Cohesive powers of \texorpdfstring{$\omega$}{omega} over \texorpdfstring{$\Delta_2$}{Delta\_2} cohesive sets}\label{sec-CohPowDelta2}

The goal of this section is to compute a single computable colored copy $\mc{O}$ of $\omega$ such that $\prod_C \mc{O}$ is colorful for every $\Delta_2$ cohesive set $C$.  As a corollary, we obtain an infinite $\Pi_1$ set with no $\Delta_2$ cohesive subset, which is a weaker version of Corollary~\ref{cor-No-pCOH}.  We then apply Lemma~\ref{lem-Shuffle} to show that there are computable copies of $\omega$ whose cohesive powers over $\Delta_2$ cohesive sets have order-types of the form $\omega + \text{various shuffles}$.

\begin{Theorem}\label{thm-ColorsDenseNonstd}
There is a computable colored copy $\mc{O}$ of $\omega$ such that $\prod_C \mc{O}$ is colorful whenever $C$ is a $\Delta_2$ cohesive set.
\end{Theorem}

\begin{proof}
We compute a colored copy $\mc{O} = (L, \Nb; \prec_\mc{L}, F)$ of $\omega$ so that $\prod_C \mc{O}$ is colorful whenever $C$ is a $\Delta_2$ cohesive set.  We take the domain of the underlying linear order $\mc{L} = (L; \prec_\mc{L})$ to be $L = \Nb$.  Given a cohesive set $C$, recall that an element $[\varphi]$ of $\prod_C \mc{L}$ is non-standard if and only if $\lim_{n \in C} \varphi(n) = \infty$ by Lemma~\ref{lem-NonstdUnbdd}.

The goal is to arrange that
\begin{align*}
\forae n \in C\; \Bigl( &\psi(n)\da \prec_\mc{L} \varphi(n)\da\\
&\Imp\; \forall d \leq \max\nolimits_<\{\varphi(n), \psi(n)\}\; \exists k\; \bigl((\psi(n) \prec_\mc{L} k \prec_\mc{L} \varphi(n)) \,\land\, (F(k) = d) \bigr)\Bigr)\tag{$*$}\label{eq-MakeDense}
\end{align*}
whenever $C$ is a $\Delta_2$ cohesive set and $\varphi$ and $\psi$ are partial computable functions with $C \subseteq^* \dom(\varphi)$, $C \subseteq^* \dom(\psi)$, and $\lim_{n \in C}\varphi(n) = \lim_{n \in C}\psi(n) = \infty$.

Suppose we have arranged that $\mc{L} \iso \omega$, suppose that $C$ is a $\Delta_2$ cohesive set, and suppose that $[\varphi]$ and $[\psi]$ are non-standard elements of $\prod_C \mc{L}$ with $[\psi] \prec_{\prod_C \mc{L}} [\varphi]$.  Then $C \subseteq^* \dom(\varphi)$, $C \subseteq^* \dom(\psi)$, and $\lim_{n \in C}\varphi(n) = \lim_{n \in C}\psi(n) = \infty$.  Further suppose that we have achieved~\eqref{eq-MakeDense} for $C$, $\varphi$, and $\psi$.  Fix any color $d$, and let $\delta$ be the constant function with value $d$.  Partially compute a function $\theta(n)$ by searching for a $k$ with $\psi(n) \prec_\mc{L} k \prec_\mc{L} \varphi(n)$ and $F(k) = d$ and by letting $\theta(n)$ be the first such $k$ if there is one.  Property~\eqref{eq-MakeDense} and the fact that $\lim_{n \in C}\varphi(n) = \lim_{n \in C}\psi(n) = \infty$ ensure that there is such a $k$ for almost every $n \in C$.  Therefore $C \subseteq^* \dom(\theta)$, $[\psi] \prec_{\prod_C \mc{L}} [\theta] \prec_{\prod_C \mc{L}} [\varphi]$, and $F^{\prod_C \mc{O}}([\theta]) = \llb \delta \rrb$.  This shows that for every solid color $\llb \delta \rrb$, there is an element $[\theta]$ of $\prod_C \mc{L}$ between $[\psi]$ and $[\varphi]$ with color $\llb \delta \rrb$.  Thus $\prod_C \mc{O}$ is colorful because it satisfies Definition~\ref{def-PowColorful} item~\ref{it-ColorfulStd} and $C$ is $\Delta_2$.  Therefore, achieving~\eqref{eq-MakeDense} suffices to prove the theorem, again provided we also arrange that $\mc{L} \iso \omega$.

Let $g \colon \Nb^3 \to \{0,1\}$ be a uniform sequence containing all $\Delta_2$-approximations as in Definition~\ref{def-UniformDelta2Approx} and the discussion following it.  Write $g_e(n,s)$ in place of $g(e,n,s)$ for all $e$, $n$, and $s$.  Let $(A^{i,0}, A^{i,1})_{i \in \Nb}$ be a computable sequence of staggered partitions as in Definition~\ref{def-StagPart} and the discussion following it.

The construction acts when quadruples $q = \la \ell, r, e, N \ra$ meet certain conditions.  Think of $\la \ell, r, e, N \ra$ as coding a pair $(\varphi_\ell, \varphi_r)$ of partial computable functions along with a guess that $g_e$ approximates a $\Delta_2$ cohesive set and, moreover, that $N$ is a threshold by which certain cohesive behavior begins.  For notational convenience, we often write $\la x, N \ra$ for $\la \ell, r, e, N \ra$, where $x = \la \ell, r, e \ra$.  In this notation, we still call $\la x, N \ra$ a `quadruple' because $x$ codes a triple.

To each triple $x = \la \ell, r, e \ra$, assign partitions $(A^{2x, 0}, A^{2x, 1})$ and $(A^{2x + 1, 0}, A^{2x + 1, 1})$.  Notice that if $C$ is a cohesive set with $C \subseteq^* \dom(\varphi_\ell)$ and $C \subseteq^* \dom(\varphi_r)$, then there is a pair of sides $(a, b) \in \{0,1\} \times \{0,1\}$ of the partitions $(A^{2x, 0}, A^{2x, 1})$ and $(A^{2x + 1, 0}, A^{2x + 1, 1})$ such that $\forae n \in C\; (\varphi_\ell(n) \in A^{2x, a})$ and $\forae n \in C\; (\varphi_r(n) \in A^{2x + 1, b})$.

The goal of a quadruple of the form $\la x, N \ra = \la \ell, r, e, N \ra$ is to attempt to satisfy~\eqref{eq-MakeDense} for $\varphi = \varphi_r$, $\psi = \varphi_\ell$, and the $\Delta_2$ cohesive set approximated by $g_e$.  However, it only needs to succeed if $g_e$ really does approximate a cohesive set $C$ and there is a pair $(a, b) \in \{0,1\} \times \{0,1\}$ such that for all $n \geq N$ with $n \in C$, $\varphi_\ell(n) \in A^{2x, a}$ and $\varphi_r(n) \in A^{2x + 1, b}$.  If ever it looks like $\varphi_\ell(n) \prec_\mc{L} \varphi_r(n)$ for an $n \geq N$ with $n \in C$ but there are not elements of every color $d \leq \max_<\{\varphi_\ell(n), \varphi_r(n)\}$ in the interval $(\varphi_\ell(n), \varphi_r(n))_\mc{L}$, then $\la x, N \ra$ needs to add elements of the missing colors to the interval.  The difficulty is that we must produce a linear order of type $\omega$, and therefore we can only place finitely many elements $\prec_\mc{L}$-below any given element.  Priority prevents certain quadruples from adding elements $\prec_\mc{L}$-below certain other elements.  Even so, it may still be that a quadruple $q_0$ adds an element $k_0$, which induces a quadruple $q_1 \leq q_0$ to add an element $k_1 \prec_\mc{L} k_0$, which induces a quadruple $q_2 \leq q_1$ to add another element $k_2 \prec_\mc{L} k_0$, and so on.  It may even be that $q_0 = q_1 = q_2 = \cdots$.  To see how this could happen, suppose that quadruple $q_0 = \la \ell_0, r_0, e_0, N \ra$ adds an element $k_0$ to the interval $(\varphi_{\ell_0}(n), \varphi_{r_0}(n))_\mc{L}$.  Later, it could be that a quadruple $q_1 = \la \ell_1, r_1, e_1, M \ra \leq q_0$ wants to add an element $k_1$ to the interval $(\varphi_{\ell_1}(m), \varphi_{r_1}(m))_\mc{L}$, but $\varphi_{r_1}(m) = k_0$.  Adding $k_1$ to $(\varphi_{\ell_1}(m), \varphi_{r_1}(m))_\mc{L}$ would thus result in defining $k_1 \prec_\mc{L} k_0$.

To avoid the sort of behavior indicated above, when a quadruple $q$ wants to add elements to $\mc{L}$, it looks at the initial triples $y$ of higher-or-equal priority requirements $\la y, M \ra = \la \ell, r, e, M \ra \leq q$, supposes that $g_e$ approximates a $\Delta_2$ cohesive set $C_e$, and tries to choose the elements that it adds to avoid each $\varphi_\ell(C_e)$ and $\varphi_r(C_e)$.  To do this, for each such $y = \la \ell, r, e \ra$, $q$ looks at the most recent guess of sides $(a_y, b_y)$ such that $\varphi_\ell(C_e) \subseteq^* A^{2y, a_y}$ and $\varphi_r(C_e) \subseteq^* A^{2y + 1, b_y}$ made by any $\la y, M \ra \leq q$ for this $y$.  Then $q$ chooses the elements it adds to $\mc{L}$ from the set
\begin{align*}
\bigcap_{\la y, M \ra \leq q} A^{2y, 1-a_y} \cap A^{2y+1, 1-b_y}.
\end{align*}
By choosing elements from the opposite sides of the partitions, $q$ attempts to avoid adding elements from the sets $\varphi_\ell(C_e)$ and $\varphi_r(C_e)$ corresponding to higher-or-equal priority quadruples.  The staggering of the partitions ensures that there are infinitely many elements for $q$ to choose among.

If for $x = \la \ell, r, e \ra$, $g_e$ really does approximate a $\Delta_2$ cohesive set $C$ with $C \subseteq^* \dom(\varphi_\ell)$ and $C \subseteq^* \dom(\varphi_r)$, then any quadruple $\la x, N \ra$ with sufficiently large $N$ eventually settles on the correct guess of sides $(a, b)$ such that $\varphi_\ell(C) \subseteq^* A^{2x, a}$ and $\varphi_r(C) \subseteq^* A^{2x + 1, b}$.  However, just as in the proof of Theorem~\ref{thm-NoCOH-helper}, many $g_e$ do not approximate $\Delta_2$ sets at all, let alone $\Delta_2$ cohesive sets.  If $g_e$ is particularly ill-behaved, then a quadruple $\la x, N \ra = \la \ell, r, e, N \ra$ with this $e$ may change its guess $(a, b)$ concerning the sides of the partitions $(A^{2x, 0}, A^{2x, 1})$ and $(A^{2x + 1, 0}, A^{2x + 1, 1})$ infinitely often.  This makes it impossible for lower priority quadruples to predict which elements to avoid adding to $\mc{L}$ in order to avoid inciting a reaction from quadruple $\la x, N \ra$.  To combat this problem, we again count how many times $\la x, N \ra$ changes its guess.  The more $\la x, N \ra$ changes its guess, the more elements of $\mc{L}$ we prevent $\la x, N \ra$ from adding elements $\prec_\mc{L}$-below.  This makes it safe for lower priority quadruples to add elements where $\la x, N \ra$ cannot.

We now describe the construction in full detail.  Define $\prec_\mc{L}$ and $F$ in stages.  By the end of stage $s$, $\prec_\mc{L}$ will have been defined on $L_s \times L_s$ and $F$ will have been defined on $L_s$ for some finite $L_s \supseteq \{0, 1, \dots, s\}$.

Say that quadruple $q = \la x, N \ra = \la \ell, r, e, N \ra$ is \emph{active} at stage $s$ if there is a pair of sides $(a, b) \in \{0,1\} \times \{0,1\}$ of the partitions $(A^{2x, 0}, A^{2x, 1})$ and $(A^{2x + 1, 0}, A^{2x + 1, 1})$ along with a witness $w$ with $N < w < s$ meeting the following conditions.
\begin{enumerate}[(a)]
\item $g_e(w,s) = 1$.

\medskip

\item For all $m$ with $N \leq m \leq w$ and $g_e(m,s) = 1$:
\begin{itemize}
\item $\varphi_{\ell,s}(m)\da \in A^{2x, a}$, and

\smallskip

\item $\varphi_{r,s}(m)\da \in A^{2x +1, b}$.
\end{itemize}
\end{enumerate}
Notice that if $q$ is active at stage $s$, then there is a unique pair $(a,b)$ for which there is a $w$ meeting these conditions.  Call this pair $(a,b)$ the $\emph{active sides}$, and call the largest witness $w$ (with $N < w < s$) the \emph{activity witness} for $q$ at stage $s$.  To each quadruple $q$, we associate a counter $\ct(q)$ that counts the number of times that the active sides of $q$ change.

At stage $0$, set $L_0 = \{0\}$ with $F(0) = 0$.  Initialize $\ct(q) = 0$ for every quadruple $q$.

At stage $s > 0$, initially set $L_s = L_{s-1}$.  If $s \notin L_s$, then add $s$ to $L_s$, define it to be the $\prec_\mc{L}$-maximum element of $L_s$, and define $F(s) = 0$.  Then consider the active quadruples $q < s$.  If this is the first stage at which $q$ is active or if the active sides of $q$ are different than they were at the previous stage at which $q$ was active, then update the counter $\ct(q)$ to $\ct(q) + 1$.

The quadruple $q = \la \ell, r, e, N \ra$ \emph{demands action} at stage $s$ if it is active with active sides $(a,b)$ and activity witness $w$ and there is a least \emph{action input} $n$ with $N \leq n \leq w$ meeting the following conditions.

\begin{enumerate}[(1)]
\item\label{it:conv} $g_e(n,s) = 1$.

\medskip

\item\label{it:gap} $\varphi_\ell(n), \varphi_r(n) \in L_s$ and $\varphi_\ell(n) \prec_\mc{L} \varphi_r(n)$, but there is a $d \leq \max_<\{\varphi_\ell(n), \varphi_r(n)\}$ for which there is no $k \in L_s$ with $\varphi_\ell(n) \prec_\mc{L} k \prec_\mc{L} \varphi_r(n)$ and $F(k) = d$.

\medskip

\item\label{it:priority} $\varphi_\ell(n)$ is $\preceq_\mc{L}$-above all of $0, 1, \dots, \max_<\{q, \ct(q)\}$.
\end{enumerate}
If $q = \la \ell, r, e, N \ra$ demands action with action input $n$, then let 
\begin{align*}
S &=\bigl\{p \leq q : \text{$\varphi_r(n)$ is $\prec_\mc{L}$-above all of $0, 1, \dots, \textstyle{\max_<}\{p, \ct(p)\}$}\bigr\}\\
R &= \bigl\{y : \exists M (\la y, M \ra \in S)\bigr\}.
\end{align*}
The idea is that $S$ consists of the higher-or-equal priority quadruples $p \leq q$ that are currently permitted to add elements $\prec_\mc{L}$-below $\varphi_r(n)$, and that $R$ consists of the initial triples of each quadruple in $S$.  For each $y \in R$, let $w$ be the greatest activity witness yet achieved by any $\la y, M \ra \in S$, let $t$ be the most recent stage at which $w$ was achieved, let $M$ be least such that $\la y, M \ra \in S$ achieved activity witness $w$ at stage $t$, and let $(a_y, b_y)$ be the corresponding active sides.  If no $\la y, M \ra \in S$ has yet been active, then let $(a_y, b_y) = (0, 0)$.  Let $c = \max_<\{\varphi_\ell(n), \varphi_r(n)\}$, and let $k_0 < k_1 < \cdots < k_c$ be the $c+1$ least members of
\begin{align*}
\bigcap_{y \in R}\left(A^{2y, 1-a_y} \cap A^{2y+1, 1-b_y}\right) \setminus L_s,
\end{align*}
which exist because the intersection is infinite and $L_s$ is finite.  Add $k_0, \dots, k_c$ to $L_s$ and place them immediately $\prec_\mc{L}$-below $\varphi_r(n)$.  That is, let $v \in L_s$ be the current $\prec_\mc{L}$-greatest element of the interval $(\varphi_\ell(n), \varphi_r(n))_\mc{L}$ (or $v = \varphi_\ell(n)$ if the interval is empty), and set
\begin{align*}
\varphi_\ell(n) \preceq_\mc{L} v \prec_\mc{L} k_0 \prec_\mc{L} \cdots \prec_\mc{L} k_c \prec_\mc{L} \varphi_r(n).
\end{align*}
Also set $F(k_i) = i$ for each $i \leq c$, and say that \emph{$q$ has acted and added $k$'s}.  This completes the construction.

The constructed $\mc{L}$ is a computable linear order.  We show that $\mc{L} \iso \omega$ by showing that for each $u$, there are only finitely many elements $\prec_\mc{L}$-below $u$.

Fix $u$, and note that $u$ appears in $L_s$ at stage $s = u$ at the latest.  Consider the evolution of the construction at stages $s > u$.

\begin{ClaimC}\label{claim-AboveU}
Suppose that quadruple $q$ acts and adds elements to $L_s$ at stage $s > u$ and either $q \geq u$ or $\ct(q) \geq u$ at stage $s$.  Then the elements added by the action of $q$ are $\prec_\mc{L}$-above $u$.
\end{ClaimC}

\begin{proof}[Proof of Claim]
If $q$ acts at stage $s > u$ with action input $n$ and either $q \geq u$ or $\ct(q) \geq u$ at stage $s$, then $u \leq \max_<\{q, \ct(q)\}$ at stage $s$.  Thus it must be that $u \preceq_\mc{L} \varphi_\ell(n) \prec_\mc{L} \varphi_r(n)$ by condition~\ref{it:priority}.  In this case, the action adds elements to $L_s$ and places them in the interval $(\varphi_\ell(n), \varphi_r(n))_\mc{L}$ and hence places them $\prec_\mc{L}$-above $u$.
\end{proof}

It follows from Claim~\ref{claim-AboveU} that no quadruple $q \geq u$ acts to add elements $\prec_\mc{L}$-below $u$ at stages $s > u$.  Thus to show that there are only finitely many elements $\prec_\mc{L}$-below $u$, it suffices to show that each quadruple $q < u$ only ever acts to add finitely many elements $\prec_\mc{L}$-below $u$.  The following Claims~\ref{claim-AgreeSides}--\ref{claim-AddGoodK} aid this analysis.

\begin{ClaimC}\label{claim-AgreeSides}
Let $x = \la \ell, r, e \ra$, and consider quadruples $\la x, M_0 \ra$ and $\la x, M_1 \ra$ with $M_0 < M_1$.  Suppose that for each $i \in \{0,1\}$, $\la x, M_i \ra$ achieves arbitrarily large activity witnesses and eventually settles on a pair of active sides $(a_i, b_i)$.  Then $(a_0, b_0) = (a_1, b_1)$.
\end{ClaimC}

\begin{proof}[Proof of Claim]
We show that at infinitely many stages $t$, both $\la x, M_0 \ra$ and $\la x, M_1 \ra$ are active with active sides $(a_0, b_0)$.  It follows that $(a_0, b_0) = (a_1, b_1)$ because $\la x, M_1 \ra$ changes its active sides only finitely often.  To this end, given any $s$, let $t > s$ be a stage at which $\la x, M_0 \ra$ is active with active sides $(a_0, b_0)$ and activity witness $w > M_1$.  Such a stage $t$ exists because $\la x, M_0 \ra$ achieves arbitrarily large activity witnesses.  Then $g_e(w,t) = 1$, and also $\varphi_{\ell, t}(m)\da \in A^{2x,a_0}$ and $\varphi_{r, t}(m)\da \in A^{2x+1,b_0}$ for all $m$ with $M_0 \leq m \leq w$ and $g_e(m,t) = 1$.  Therefore $(a_0, b_0)$ and $w$ also witness that $\la x, M_1 \ra$ is active at stage $t$ with active sides $(a_0, b_0)$ because $M_0 < M_1 < w$.
\end{proof}

\begin{ClaimC}\label{claim-ActOnce}
For a given quadruple $q$, each number $n$ can be the action input for $q$ at most once.
\end{ClaimC}

\begin{proof}[Proof of Claim]
Suppose that $q = \la \ell, r, e, N \ra$ demands action with action input $n$ at some stage $s$.  Then $q$ adds elements of every color $i \leq \max_<\{\varphi_\ell(n), \varphi_r(n)\}$ to $L_s$ and places them $\prec_\mc{L}$-between $\varphi_\ell(n)$ and $\varphi_r(n)$.  Thus condition~\ref{it:gap} is never again satisfied for $q$ with action input $n$ at any stage $t > s$.
\end{proof}

Let $q_0 < u$, and assume inductively that there is a stage $s_0 > u$ such that no quadruple $p < q_0$ acts to add elements $k \prec_\mc{L} u$ after stage $s_0$.  We show that quadruple $q_0$ also adds only finitely many elements $\prec_\mc{L}$-below $u$.  There are three cases:  (i) the activity witnesses for $q_0$ are bounded, (ii) quadruple $q_0$ changes its active sizes at least $u$ many times, and (iii) quadruple $q_0$ achieves arbitrarily large activity witnesses and changes its active sides fewer than $u$ many times.

In order for $q_0$ to act and add elements at some stage $s$, it must be that $q_0$ is active at stage $s$ with some activity witness $w$, and there must be an action input $n$ with $n \leq w$.  Therefore, if we are in case~(i) and the activity witnesses for $q_0$ are always $\leq W$ for some fixed $W$, then $q_0$ only acts to add elements finitely often.  This is because when $q_0$ acts to add elements, it must use an action input $n \leq W$, and, by Claim~\ref{claim-ActOnce}, $q_0$ can use each $n \leq W$ as an action input at most once.  

If we are in case~(ii) and quadruple $q_0$ changes its active sides at least $u$ many times, then $\ct(q_0)$ is incremented at least $u$ many times, so there is a stage $s$ such that $\ct(q_0) \geq u$ at all stages $t \geq s$.  Thus if $q_0$ acts at a stage $t \geq s$, then the elements it adds at that stage are $\prec_\mc{L}$-above $u$ by Claim~\ref{claim-AboveU}.  Thus $q_0$ only ever adds finitely many elements $\prec_\mc{L}$-below $u$ in this case as well.

For the remaining case~(iii), suppose that our quadruple $q_0 = \la x_0, N_0 \ra = \la \ell_0, r_0, e_0, N_0 \ra$ achieves arbitrarily large activity witnesses and changes its active sides fewer than $u$ many times.  Let $\fct_0 < u$ denote the final value of $\ct(q_0)$, and let $(a_0, b_0)$ be the final active sides of $q_0$.  Let $v = \max_{\prec_\mc{L}}\{0, 1, \dots, \max_<\{q_0, \fct_0\}\}$.  We claim that eventually every element $k$ added to $L_s$ with  $v \prec_\mc{L} k \prec_\mc{L} u$ is in $A^{2x_0, 1-a_0} \cap A^{2x_0+1, 1-b_0}$.

\begin{ClaimC}\label{claim-AddGoodK}
There is a stage $s_1 \geq s_0$ such that whenever an element $k$ is added to $L_s$ and $v \prec_\mc{L} k \prec_\mc{L} u$ is defined at some stage $s \geq s_1$, we have that $k \in A^{2x_0, 1-a_0} \cap A^{2x_0+1, 1-b_0}$.
\end{ClaimC}

\begin{proof}[Proof of Claim]
By the choice of $s_0$, the fact that $s_0 > u$, and Claim~\ref{claim-AboveU}, we already know that no quadruple $p$ with either $p < q_0$ or $p \geq u$ acts to add elements $\prec_\mc{L}$-below $u$ after stage $s_0$.  Thus we need only consider the behavior of quadruples $p$ with $q_0 \leq p < u$.

Recall that our quadruple $q_0$ is $q_0 = \la x_0, N_0 \ra$.  Consider all quadruples $\la x_0, M \ra < u$ with this same $x_0$, including quadruple $q_0$.  Let $t_0$ and $s_1$ with $s_1 > t_0 > s_0$ be large enough so that the following hold.
\begin{itemize}
\item Each $\la x_0, M \ra < u$ that eventually achieves $\ct(\la x_0, M \ra) \geq u$ has done so by stage $t_0$.

\medskip

\item Each $\la x_0, M \ra < u$ that changes its active sides only finitely often has settled on its final active sides and on its final value of $\ct(\la x_0, M \ra)$ by stage $t_0$.

\medskip

\item For each $\la x_0, M \ra < u$ that does not achieve arbitrarily large activity witnesses, quadruple $q_0$ has, by stage $t_0$, achieved an activity witness larger than the maximum activity witness ever achieved by $\la x_0, M \ra$.

\medskip

\item By stage $s_1$, quadruple $q_0$ has achieved an activity witness larger than all the activity witnesses achieved by all $\la x_0, M \ra < u$ at stages $t \leq t_0$.
\end{itemize}
Such $t_0$ and $s_1$ exist because $q_0$ achieves arbitrarily large activity witnesses.

Suppose that quadruple $q = \la \ell, r, e, N \ra$ with $q_0 \leq q < u$ acts and adds $k$'s in the interval $(v, u)_{\mc L}$ at a stage $s \geq s_1$.  Then it must be that $v \prec_\mc{L} \varphi_r(n) \preceq_\mc{L} u$, where $n$ is the action input for $q$ at stage $s$, because the action of $q$ adds $k$'s immediately $\prec_\mc{L}$-below $\varphi_r(n)$.  In particular, at stage $s$, $\varphi_r(n)$ is $\prec_\mc{L}$-above all of $0, 1, \dots, \max_<\{q_0, \ct(q_0)\}$.  This means that $q_0 \in S$ and $x_0 \in R$, where $S$ and $R$ are the sets used by $q$ when it acts at stage $s$.  Thus when $q$ acts at stage $s$, it chooses active sides $(a_{x_0}, b_{x_0})$ corresponding to $x_0$.  To do this, the action of $q$ finds the greatest activity witness $w$ yet achieved by any $\la x_0, M \ra \in S$, the most recent stage $t$ at which $w$ was achieved, the least $M$ such that $\la x_0, M \ra \in S$ achieved activity witness $w$ at stage $t$, and then takes $(a_{x_0}, b_{x_0})$ to be the active sides of $\la x_0, M \ra$ at stage $t$.  We show that $(a_{x_0}, b_{x_0}) = (a_0, b_0)$.

Consider the quadruple $\la x_0, M \ra$ used to choose the active sides $(a_{x_0}, b_{x_0})$ corresponding to $x_0$ during the action of $q$ at stage $s$ as described above.  Note that $\la x_0, M \ra \leq q < u$.  Quadruple $\la x_0, M \ra$ changes its active sides only finitely often.  If $\la x_0, M \ra$ changes its active sides infinitely often, then eventually $\la x_0, M \ra$ achieves $\ct(\la x_0, M \ra) \geq u$ and hence has done so by stage $s > t_0$.  Therefore, at stage $s$, $\varphi_r(n)$ is not $\prec_\mc{L}$-above all of $0, 1, \dots, \max_<\{\la x_0, M \ra, \ct(\la x_0, M \ra)\}$ because $\varphi_r(n) \preceq_\mc{L} u$ but $\ct(\la x_0, M \ra) \geq u$.  This implies that $\la x_0, M \ra \notin S$, which is a contradiction.

Quadruple $\la x_0, M \ra$ achieves arbitrarily large activity witnesses.  If not, then by stage $s > t_0$ quadruple $q_0 = \la x_0, N_0 \ra$ has already achieved an activity witness greater than any activity witness ever achieved by $\la x_0, M \ra$.  Quadruple $q_0$ is in $S$, so the greatest activity witness achieved by a quadruple of the form $\la x_0, \wh{M} \ra$ in $S$ by stage $s$ was not achieved by $\la x_0, M \ra$.  This contradicts that $q$ uses $\la x_0, M \ra$ to choose $(a_{x_0}, b_{x_0})$ at stage $s$.

The most recent stage $t \leq s$ at which quadruple $\la x_0, M \ra$ achieves activity witness $w$ must satisfy $t > t_0$.  If $t \leq t_0$, then, by choice of $s_1$, quadruple $q_0 = \la x_0, N_0 \ra$ has already achieved an activity witness greater than $w$ by stage $s > s_1$.  This contradicts that $w$ is the greatest activity witness achieved by a quadruple of the form $\la x_0, \wh{M} \ra$ in $S$ by stage $s$.

It now follows that $(a_{x_0}, b_{x_0})$ are the final active sides of $\la x_0, M \ra$.  Quadruple $\la x_0, M \ra$ changes its active sides finitely often, so it settles on its final active sides by stage $t_0$ by choice of $t_0$.  The pair $(a_{x_0}, b_{x_0})$ is the active sides of $\la x_0, M \ra$ at the most recent stage $t \leq s$ at which $\la x_0, M \ra$ achieved activity witness $w$.  We showed that $t > t_0$, so $(a_{x_0}, b_{x_0})$ must be the final active sides of $\la x_0, M \ra$.

Recall that $(a_0, b_0)$ are the final active sides of $q_0$.  Quadruples $\la x_0, M \ra$ and $q_0 = \la x_0, N_0 \ra$ both achieve arbitrarily large activity witnesses and eventually settle on their active sides.  Therefore $\la x_0, M \ra$ and $q_0$ settle on the same active sides by Claim~\ref{claim-AgreeSides}.  Thus $(a_{x_0}, b_{x_0}) = (a_0, b_0)$.  Therefore, the $k$'s that the action of $q$ chooses to add to $L_s$ are in $A^{2x_0, 1-a_0} \cap A^{2x_0+1, 1-b_0}$.

Thus we have found an $s_1 \geq s_0$ such that whenever an element $k$ is added to $L_s$ and $v \prec_\mc{L} k \prec_\mc{L} u$ is defined at a stage $s \geq s_1$, it is on account of a quadruple $q$ with $q_0 \leq q < u$ and we have that $k \in A^{2x_0, 1-a_0} \cap A^{2x_0+1, 1-b_0}$.
\end{proof}

We may now show that our quadruple $q_0 = \la x_0, N_0 \ra = \la \ell_0, r_0, e_0, N_0\ra$ adds only finitely many elements $k \prec_\mc{L} u$.  Let $s_1$ be the stage from Claim~\ref{claim-AddGoodK}, and additionally assume that $q_0$ has settled on its final active sides $(a_0, b_0)$ and that $\ct(q_0)$ has reached its final value $\fct_0$ by stage $s_1$.  Recall that $v = \max_{\prec_\mc{L}}\{0, 1, \dots, \max_<\{q_0, \fct_0\}\}$.  By condition~\ref{it:priority}, every $k$ that $q_0$ adds to $L_s$ at a stage $s \geq s_1$ satisfies $v \prec_\mc{L} k$.  So suppose that $q_0$ acts at some stage $s \geq s_1$, adds an element $k$ to $L_s$, and defines $k \prec_\mc{L} u$, in which case it also defines $v \prec_\mc{L} k$.  Then at stage $s$, $q_0$ is active with active sides $(a_0, b_0)$, and it acts with action input $n$, where $\varphi_{\ell_0}(n) = i$ for some $i \in A^{2x_0, a_0}$, $\varphi_{r_0}(n) = j$ for some $j \in A^{2x_0+1, b_0}$, and $v \preceq_\mc{L} i \prec_\mc{L} j \preceq_\mc{L} u$.  The action then places $k$'s of each color $d \leq \max_<\{i,j\}$ in the interval $(i, j)_\mc{L}$.  If $q_0$ acts again at some later stage $t > s$ with some action input $m$, then again $\varphi_{\ell_0}(m) \in A^{2x_0, a_0}$ and $\varphi_{r_0}(m) \in A^{2x_0+1, b}$.  However, it cannot again be that $\varphi_{\ell_0}(m) = i$ and $\varphi_{r_0}(m) = j$ because condition~\ref{it:gap} would fail in this situation.  Thus when adding a number $k \prec_\mc{L} u$, the action input $n$ used by $q_0$ specifies a pair $(i,j) = (\varphi_{\ell_0}(n), \varphi_{r_0}(n)) \in A^{2x_0, a_0} \times A^{2x_0+1, b_0}$ with $v \preceq_\mc{L} i \prec_\mc{L} j \preceq_\mc{L} u$, and each such pair can be specified by $q_0$ at most once.  By Claim~\ref{claim-AddGoodK}, every element added to the interval $(v, u)_\mc{L}$ after stage $s_1$ is in $A^{2x_0, 1-a_0} \cap A^{2x_0+1, 1-b_0}$.  Therefore, there are only finitely many pairs $(i,j) \in A^{2x_0, a_0} \times A^{2x_0+1, b_0}$ with $v \preceq_\mc{L} i \prec_\mc{L} j \preceq_\mc{L} u$, and therefore quadruple $q_0$ can only add finitely many elements $k \prec_\mc{L} u$.  This completes case~(iii) and thus completes the proof that $\mc{L} \iso \omega$.

To complete the proof of the theorem, we must show that~\eqref{eq-MakeDense} is satisfied whenever $C$ is a $\Delta_2$ cohesive set and $\varphi$ and $\psi$ are partial computable functions with $C \subseteq^* \dom(\varphi)$, $C \subseteq^* \dom(\psi)$, and $\lim_{n \in C}\varphi(n) = \lim_{n \in C}\psi(n) = \infty$.

\begin{ClaimC}\label{claim-GoodQuad}
Suppose that quadruple $q = \la x, N \ra =  \la \ell, r, e, N \ra$ and pair $(a,b)$ are such that
\begin{itemize}
\item $g_e$ is a $\Delta_2$-approximation to an infinite $\Delta_2$ set $C$,

\smallskip

\item $\forall m \geq N\; (m \in C \;\imp\; \varphi_\ell(m)\da \in A^{2x, a})$, and

\smallskip

\item $\forall m \geq N\; (m \in C \;\imp\; \varphi_r(m)\da \in A^{2x+1, b})$.
\end{itemize}
Then $q$ achieves arbitrarily large activity witnesses, and $q$ eventually settles on active sides $(a,b)$.
\end{ClaimC}

\begin{proof}[Proof of Claim]
Given any number $W$, let $w$ be the least number with $w > W$ and $w \in C$.  Let $s > w$ be large enough so that $g_e(m,s) = C(m)$ for all $m \leq w$ and so that $\varphi_{\ell, s}(m)\da \in A^{2x, a}$ and $\varphi_{r, s}(m)\da \in A^{2x+1, b}$ for all $m \in C$ with $N \leq m \leq w$.  Then $q$ is active at stage $s$ with activity witness $w$ or greater.  Thus $q$ achieves arbitrarily large activity witnesses.

Let $n$ be the least number with $n > N$ and $n \in C$.  Let $s_0$ be large enough so that $\forall m \leq n\; \forall s \geq s_0\; (g_e(m,s) = C(m))$, $\varphi_{\ell, s_0}(n)\da \in A^{2x,a}$, and $\varphi_{r, s_0}(n)\da \in A^{2x+1,b}$.  Then if $q$ is active at a stage $s > s_0$, it must use an activity witness $w \geq n$, in which case its active sides must be $(a,b)$ because $g_e(n,s) = 1$, $\varphi_{\ell, s}(n)\da \in A^{2x,a}$, and $\varphi_{r, s}(n)\da \in A^{2x+1,b}$.  That is, $q$ has active sides $(a,b)$ whenever it is active at a stage later than $s_0$.  Thus $q$ eventually settles on active sides $(a,b)$.
\end{proof}

Let $C$ be a $\Delta_2$ cohesive set, and let $\varphi$ and $\psi$ be partial computable functions with $C \subseteq^* \dom(\varphi)$, $C \subseteq^* \dom(\psi)$, and $\lim_{n \in C}\varphi(n) = \lim_{n \in C}\psi(n) = \infty$.  We show that~\eqref{eq-MakeDense} holds for $C$, $\varphi$, and $\psi$.  Assume that $\forae n \in C\; (\psi(n) \prec_\mc{L} \varphi(n))$, for otherwise~\eqref{eq-MakeDense} vacuously holds.  Let $\ell$ and $r$ be such that $\varphi_\ell = \psi$ and $\varphi_r = \varphi$.  Let $e$ be such that $g_e$ is a $\Delta_2$-approximation to $C$.  Let $x = \la \ell, r, e \ra$.  By cohesiveness, let $(a,b)$ and $N$ be such that, for all $n \in C$ with $n \geq N$, $\varphi_\ell(n) \in A^{2x, a}$ and $\varphi_r(n) \in A^{2x+1, b}$.  Let $q$ be the quadruple $q = \la x, N \ra = \la \ell, r, e, N \ra$.  Quadruple $q$ and pair $(a,b)$ satisfy the hypotheses of Claim~\ref{claim-GoodQuad}, so the active sides of $q$ eventually settle on $(a,b)$ and therefore $\ct(q)$ also reaches a final value $\fct$.  Let $v = \max_{\prec_\mc{L}}\{0, 1, \dots, \max_<\{q, \fct\}\}$.  Let $n_0 \geq N$ be large enough so that $v \preceq_\mc{L} \varphi_\ell(n)$ whenever $n \geq n_0$ and $n \in C$.  Such an $n_0$ exists because $\lim_{n \in C}\varphi_\ell(n) = \infty$, but there are only finitely many elements $\prec_\mc{L}$-below $v$ because $\mc{L} \iso \omega$.

Suppose that $n \in C$ and $n \geq n_0$, and furthermore suppose for a contradiction that there is a $d \leq \max_<\{\varphi_\ell(n), \varphi_r(n)\}$ such that there is no $k$ with $\varphi_\ell(n) \prec_\mc{L} k \prec_\mc{L} \varphi_r(n)$ and $F(k) = d$.  Then conditions \ref{it:conv}--\ref{it:priority} hold for $n$ at all sufficiently large stages $s$, with~\ref{it:conv} holding because $n \in C$,~\ref{it:gap} holding by assumption, and~\ref{it:priority} holding by the choice of $n_0$.  By Claim~\ref{claim-GoodQuad}, quadruple $q$ achieves arbitrarily large activity witnesses.  Thus infinitely often $q$ is active with an activity witness $w > n$.  By Claim~\ref{claim-ActOnce}, each $m < n$ can be the action input for $q$ at most once.  Thus at some stage, $q$ eventually demands action with action input $n$.  The action of $q$ defines $\varphi_\ell(n) \prec_\mc{L} k \prec_\mc{L} \varphi_r(n)$ and $F(k) = d$ for some $k$, which contradicts that there is no such $k$.  This shows that~\eqref{eq-MakeDense} holds for $C$, $\varphi = \varphi_r$ and $\psi = \varphi_\ell$, which completes the proof.
\end{proof}

Theorem~\ref{thm-ColorsDenseNonstd} immediately provides a computable copy $\mc{L}$ of $\omega$ for which $\prod_C \mc{L} \iso \omega + \eta$ whenever $C$ is a $\Delta_2$ cohesive set.

\begin{Corollary}\label{cor-DenseNonstdDelta2}
There is a computable copy $\mc{L}$ of $\omega$ such that for every $\Delta_2$ cohesive set $C$, $\prod_C \mc{L} \iso \omega + \eta$.
\end{Corollary}

\begin{proof}
Let $\mc{O} = (L, \Nb; \prec_\mc{L}, F)$ be the computable colored copy of $\omega$ from Theorem~\ref{thm-ColorsDenseNonstd}, and let $\mc{L} = (L; \prec_\mc{L})$ be the underlying computable linear order of type $\omega$.  Let $C$ be a $\Delta_2$ cohesive set.  Then $\prod_C \mc{O}$ is colorful, so $\prod_C \mc{L} \iso \omega + \eta$ as explained in the discussion following Definition~\ref{def-PowColorful}.
\end{proof}

Corollary~\ref{cor-DenseNonstdDelta2} is as good as possible, in the sense that $\Delta_2$ cannot be improved to $\Pi_2$.

\begin{Proposition}\label{prop-SuccNonstdPi2}
For every computable copy $\mc{L}$ of $\omega$, there is a $\Pi_2$ cohesive set $C$ such that $\prod_C \mc{L} \niso \omega + \eta$.
\end{Proposition}

\begin{proof}
Let $\mc{L} = (L; \prec_\mc{L})$ be a computable copy of $\omega$.  Fix an $\ell \in L$.  Define total computable functions $\psi$ and $\varphi$ by
\begin{align*}
\psi(\la m, n \ra) &=
\begin{cases}
m & \text{if $m \in L$}\\
\ell & \text{if $m \notin L$}
\end{cases}\\ \\
\varphi(\la m, n \ra) &=
\begin{cases}
n & \text{if $n \in L$}\\
\ell & \text{if $n \notin L$}.
\end{cases}
\end{align*}
Let
\begin{align*}
X = \{\la m, n \ra : \text{$m, n \in L$ and $n$ is the $\prec_\mc{L}$-immediate successor of $m$}\}.
\end{align*}
Then $X$ is an infinite $\Pi_1$ set because $\mc{L}$ is a computable linear order of type $\omega$.  Therefore $X$ has a $\Pi_2$ cohesive subset $C$ as explained in the discussion following Corollary~\ref{cor-No-pCOH}.

Consider the elements $[\psi]$ and $[\varphi]$ of $\prod_C \mc{L}$.  Every element of $\mc{L}$ has exactly one $\prec_\mc{L}$-immediate successor, so for every $m$ there is at most one $n$ with $\la m, n \ra \in C$, and for every $n$ there is at most one $m$ with $\la m, n \ra \in C$.  It follows that $\psi$ and $\varphi$ are injective when restricted to $C$ and therefore that $\lim_{\la m, n \ra \in C}\psi(\la m, n \ra) = \lim_{\la m, n \ra \in C}\varphi(\la m, n \ra) = \infty$.  Thus $[\psi]$ and $[\varphi]$ are non-standard elements of $\prod_C \mc{L}$ by Lemma~\ref{lem-NonstdUnbdd}.  By the choice of $X$ and $C$, $\varphi(\la m, n \ra) = n$ is the $\prec_\mc{L}$-immediate successor of $\psi(\la m, n \ra) = m$ for every $\la m, n \ra \in C$.  Therefore $[\varphi]$ is the $\prec_{\prod_C \mc{L}}$-immediate successor of $[\psi]$ in $\prod_C \mc{L}$ by Lemma~\ref{lem-ImmedSucc}.  Thus the non-standard elements of $\prod_C \mc{L}$ are not dense, so $\prod_C \mc{L} \niso \omega + \eta$.
\end{proof}

\begin{Remark}
We obtain an infinite $\Pi_1$ set with no $\Delta_2$ cohesive subset as a consequence of Corollary~\ref{cor-DenseNonstdDelta2}.  Let $\mc{L} = (L; \prec_\mc{L})$ be a computable copy of $\omega$ such that $\prod_C \mc{L} \iso \omega + \eta$ for every $\Delta_2$ cohesive set $C$.  Let
\begin{align*}
X = \{\la m, n \ra : \text{$m, n \in L$ and $n$ is the $\prec_\mc{L}$-immediate successor of $m$}\}.
\end{align*}
as in the proof of Proposition~\ref{prop-SuccNonstdPi2}.  Then $X$ is an infinite $\Pi_1$ set.  The proof of Proposition~\ref{prop-SuccNonstdPi2} shows that if $C$ is a cohesive subset of $X$, then there is a non-standard $\prec_{\prod_C \mc{L}}$-immediate successor pair $[\psi] \prec_{\prod_C \mc{L}} [\varphi]$ in $\prod_C \mc{L}$.  Therefore, if $C$ is a cohesive subset of $X$, then $\prod_C \mc{L} \niso \omega + \eta$.  As $\prod_C \mc{L} \iso \omega + \eta$ for every $\Delta_2$ cohesive set $C$, it follows that $X$ cannot have a $\Delta_2$ cohesive subset.  So $X$ is an infinite $\Pi_1$ set with no $\Delta_2$ cohesive subset.
\end{Remark}

Finally, we reach the main result by combining Theorem~\ref{thm-ColorsDenseNonstd} with Lemma~\ref{lem-Shuffle}.

\begin{Theorem}\label{thm-BCSigma2ShuffleNew}
Let $X \subseteq \Nb \setminus \{0\}$ be a Boolean combination of $\Sigma_2$ sets, thought of as a set of finite order-types.  Then there is a computable copy $\mc{L}$ of $\omega$ such that for every $\Delta_2$ cohesive set $C$, the cohesive power $\prod_C \mc{L}$ has order-type $\omega + \bm{\sigma}(X \cup \{\omega + \zeta\eta + \omega^*\})$.  Moreover, if $X$ is finite and non-empty, then there is also a computable copy $\mc{L}$ of $\omega$ such that for every $\Delta_2$ cohesive set $C$, the cohesive power $\prod_C \mc{L}$ has order-type $\omega + \bm{\sigma}(X)$.
\end{Theorem}

\begin{proof}
By Theorem~\ref{thm-ColorsDenseNonstd}, let $\mc{O}$ be a computable colored copy of $\omega$ such that for every $\Delta_2$ cohesive set $C$, the cohesive power $\prod_C \mc{O}$ is colorful.  Then apply Lemma~\ref{lem-Shuffle} to $\mc{O}$ to get the desired computable copy $\mc{L}$ of $\omega$ in either case.
\end{proof}

Similar to \cite{CohPowJournal}*{Example 5.5}, we can define a countable collection of pairwise isomorphic (but not computably isomorphic) linear orders whose cohesive powers over $\Delta_2$ cohesive sets are pairwise non-elementarily equivalent.

\begin{Example}\label{ex-NonEquivSeq}
There are computable copies $\mc{L}^1, \mc{L}^2, \mc{L}^3, \dots$ of $\omega$ such that the cohesive powers $\prod_C \mc{L}^k$ and $\prod_D \mc{L}^m$ are not elementarily equivalent whenever $1 \leq k < m$ and $C$ and $D$ are $\Delta_2$ cohesive sets.  (We put the index in the superscript to emphasize that we mean the cohesive powers of the individual structures, not the cohesive product of the sequence.)  For each $k \geq 1$, apply Theorem~\ref{thm-BCSigma2ShuffleNew} to the set $X = \{k\}$ to get a computable copy $\mc{L}^k$ of $\omega$ such that for every $\Delta_2$ cohesive set $C$, $\prod_C \mc{L}^k \iso \omega + \bm{\sigma}(\{\bm{k}\}) \iso \omega + \bm{k}\eta$.  If $1 \leq k < m$ and $C$ and $D$ are $\Delta_2$ cohesive sets, then $\prod_C \mc{L}^k \iso \omega + \bm{k}\eta$ and $\prod_D \mc{L}^m \iso \omega + \bm{m}\eta$.  The order-types $\omega + \bm{k}\eta$ and $\omega + \bm{m}\eta$ are not elementarily equivalent because they disagree on the $\Sigma_3$ sentence expressing that there is a maximal block (in the sense of the finite condensation) of size $k$.

In fact, Theorem~\ref{thm-BCSigma2ShuffleNew} is not necessary to obtain this example.  Instead, let $\mc{L}$ be the computable copy of $\omega$ from Corollary~\ref{cor-DenseNonstdDelta2}, and let $\mc{L}^k = \bm{k}\mc{L}$ for each $k \geq 1$.  Then
\begin{align*}
\prod\nolimits_C \mc{L}^k \;\iso\; \prod\nolimits_C(\bm{k}\mc{L}) \;\iso\; \bigl( \prod\nolimits_C \bm{k} \bigr) \bigl( \prod\nolimits_C \mc{L} \bigr) \;\iso\; \bm{k}(\omega + \eta) \;\iso\; \omega + \bm{k}\eta
\end{align*}
for every $k \geq 1$ and every $\Delta_2$ cohesive set $C$.  The second isomorphism is by Theorem~\ref{thm-CohPres} item~\ref{it-ProdIso}, and $\prod_C \bm{k} \iso \bm{k}$ in the third isomorphism because $\bm{k}$ is finite.  If $\mc{A}$ is a finite computable structure and $C$ is cohesive, then, by cohesiveness, every element of $\prod_C \mc{A}$ is in the range of the canonical embedding of $\mc{A}$ into $\prod_C \mc{A}$.  Therefore $\prod_C \mc{A} \iso \mc{A}$.  Here it is apparent that the sequence $\mc{L}^1, \mc{L}^2, \mc{L}^3, \dots$ may be taken to be uniformly computable.  The sequence may be taken to be uniformly computable in the first situation too because the results of Lemma~\ref{lem-Shuffle} hold uniformly.
\end{Example}

\section*{Acknowledgments}

We give many thanks to Volodya Shavrukov for his very helpful comments on Section~\ref{sec-NoCOH}, chief among which are drawing our attention to~\cite{LermanShoreSoare} and pointing out the connections among r-maximal major subsets, $\Delta_3$ preference functions, and $\Delta_2$ r-cohesive subsets.  We also thank Carl Jockusch and Richard Shore for very helpful discussions and suggestions.  Finally, we thank the anonymous reviewers for their suggestions which helped improve the clarity of this work.

\section*{Funding}
This project was partially supported by the John Templeton Foundation grant ID 60842 and by EPSRC grant EP/T031476/1.

\bibliographystyle{amsplain}
\bibliography{PowersOfOmegaDelta2}

\vfill

\end{document}